\documentclass[a4paper, 12pt]{amsart}
\usepackage{amsmath, amsthm, amssymb, fullpage}
\usepackage{graphicx}
\usepackage{xspace}
\usepackage{enumerate}
\usepackage[colorlinks]{hyperref}

\newtheorem{thm}{Theorem}[section]
\newtheorem{lem}[thm]{Lemma}
\newtheorem*{lem*}{Lemma}

\newtheorem{rmk}[thm]{Remark}
\newtheorem{cor}[thm]{Corollary}
\newtheorem{defn}[thm]{Definition}
\newtheorem{myclaim}[thm]{Claim}
\newtheorem{prop}[thm]{Proposition}

\newtheorem*{conv}{Convention}

\newcommand{\N}{\mathbb{N}}
\newcommand{\R}{\mathbb{R}}
\newcommand{\Z}{\mathbb{Z}}
\newcommand{\T}{\mathbb{T}}
\newcommand{\mcl}{\mathcal L}
\newcommand{\bbp}{\mathbb P}

\newcommand{\E}{\mathbb E}

\newcommand{\al}{\alpha}
\newcommand{\be}{\beta}
\newcommand{\ga}{\gamma}
\newcommand{\Ga}{\Gamma}
\newcommand{\del}{\delta}

\newcommand{\ep}{\epsilon}
\newcommand{\sig}{\sigma}
\newcommand{\ka}{\kappa}
\newcommand{\lam}{\lambda}
\newcommand{\Lam}{\Lambda}
\newcommand{\Om}{\Omega}
\newcommand{\om}{\omega}

\newcommand{\tld}[1]{\tilde{#1}}
\newcommand{\mc}[1]{\mathcal{#1}}

\newcommand{\leb}{\text{Leb}}

\DeclareMathOperator{\essinf}{ess\ inf}
\DeclareMathOperator{\esssup}{ess\ sup}
\DeclareMathOperator{\var}{var}
\newcommand{\spt}{\mathop{\mathrm{supp}}}

\newcommand{\hpt}{{\mc{H}_p^t}}
\newcommand{\hptt}{{\mc{H}_p^{t'}}}

\newcommand{\hptorus}{{{H}_p^t(\T)}}
\newcommand{\paeom}{\bbp\text{-a.e. } \om \in \Om}
\newcommand{\LY}{\text{LY}}
\newcommand{\dist}{\mc{D}} 

\newcommand{\dom}{I} 
\newcommand{\nint}{\be} 

\title{Stability and approximation of random invariant densities for Lasota-Yorke map cocycles}
\author{Gary Froyland, Cecilia Gonz\'alez-Tokman and Anthony Quas}
\address[Froyland and Gonz\'alez-Tokman]{School of Mathematics and Statistics,
University of New South Wales, Sydney, NSW, 2052, Australia}
\address[Quas]{Department of Mathematics and Statistics, University of Victoria, Victoria, BC, CANADA, V8W 3R4}

\begin{document}

\begin{abstract}
We establish stability of random absolutely continuous invariant measures (acims) for cocycles of random Lasota-Yorke maps  under a variety of perturbations.
Our family of random maps need not be close to a fixed map; thus, our results can handle very general driving mechanisms.
We consider (i) perturbations via convolutions, (ii) perturbations arising from finite-rank transfer operator approximation schemes and (iii) static perturbations, perturbing to a nearby cocycle of Lasota-Yorke maps.
The former two results provide a rigorous framework for the numerical approximation of random acims using a Fourier-based approach and Ulam's method, respectively;  we also demonstrate the efficacy of these schemes.
\end{abstract}

\maketitle

\section{Introduction}

Random (or forced) dynamical systems are invaluable models of systems
exhibiting time dependence. Even though, for us, the terms random and
forced are interchangeable, we use exclusively the term random dynamical
systems (rds). These systems arise naturally in situations with
time-dependent forcing, as well as being a natural model for systems in
which some neglected or ill-understood phenomena lead to uncertainty
in the evolution. 
One particular motivation concerns transport phenomena, such as oceanic and atmospheric flow. 
Although randomness appears in the title, this work covers a variety
of situations, ranging from deterministic forcing-- for example when the driving depends
only on the time of the day-- to independent
identically distributed noise. The results of this paper deal with very
general driving systems: the conditions on the base dynamics are that it
should be stationary (i.e. have an invariant probability measure), ergodic
and invertible; in particular, no mixing properties are needed.

The predecessor works \cite{FLQ1,FLQ2,GTQuas} provide a unified framework for the study of random absolutely continuous invariant measures, exponential decay of correlations and coherent structures in random dynamical systems.
The abstract results therein, called semi-invertible Oseledets theorems, show a dynamically meaningful way of \textit{splitting} Banach spaces arising from the study of associated transfer operators into subspaces with specific growth rates.
Such results have applications in the context of random, non-autonomous or time-dependent systems, provided there are some good statistics for the time-dependence and
a \textit{quasi-compactness} type property holds.
The references above provide explicit applications in the setting of random compositions of piecewise smooth expanding maps.

Having demonstrated the existence of a splitting, it is natural to ask about its stability under different types of perturbations.
This is known to be a very difficult problem in general, even in finite dimensions, where positive results \cite{young86,LedrappierYoung} rely on absolute continuity and uniformity of the perturbations.
In the infinite-dimensional case, research has focused on transfer operators, and Perron-Frobenius operators in particular.
Considering a single map $T$ with expanding or hyperbolic properties, the transfer operator is quasi-compact in the right Banach space setup, and one can ask about the stability of eigenprojections of the transfer operator with respect to perturbations.
As the transfer operators typically preserve a non-negative cone, by the Ruelle-Perron-Frobenius theorem, there is an eigenvalue of largest magnitude, which is positive.
In the Perron-Frobenius case, this eigenvalue is 1, and the corresponding eigenprojection(s) represent invariant densities of the generating map $T$.

Numerical methods for approximating invariant densities rely on stability of the density under particular perturbations;  those induced by the numerical method.
A very common perturbation is Ulam's method, a relatively crude, but in practice extremely effective, approach.
Positive stability results in a variety of settings include \cite{BaladiIsolaSchmitt,FroylandSRB,DingZhou,BlankKeller97,MurrayThesis,KeaneMurrayYoung,Froylandrandom,Murraynonuniform}. A mechanism causing instability is described in \cite{Keller82}.
Ulam's method can also be used to estimate other non-essential spectral values \cite{FroylandCMP,BlankKeller98,BaladiHolschneider,FroylandUlamApprox}.
Stability under convolution-type perturbations is treated in \cite{BaladiYoung,AlvesVilarinho}, and \cite{BlankKellerLiverani,DemersLiverani} consider static perturbations, as well as of convolution-type.
A seminal paper in this area is \cite{KellerLiverani99}, which provides a rather general template for stability results for single maps.

Despite the considerable volume of results in the autonomous setting, only a few results are known about stability of Oseledets splittings in the non-autonomous situation
\cite{BaladiKondahSchmitt,BaladiQuenched,Bogenschutz}.
Each of these results concerns stability of Oseledets splittings for small random perturbations of a \emph{fixed} expanding map;  thus these results concern stability of \emph{non-random} eigenprojections of a fixed unperturbed transfer operator.

In contrast, we begin with a random dynamical system that possesses a (random) splitting, and demonstrate stability of this random splitting under perturbations.
Our techniques handle convolution-type perturbations (the random map experiences integrated noise), static perturbations  (the random map is perturbed to another random map), and finite-rank perturbations (stability under numerical schemes).
This paper deals with stability of the \textit{top} space of the splitting. In particular,
our results answer a question raised by Buzzi in \cite{BuzziEDC} about stability of random acims for Lasota-Yorke maps. Stability under other types of perturbations relevant for applications and numerical studies, such as Ulam and Fourier-based schemes are also be treated with our method.

The approach we take is modeled on results of Keller and Liverani \cite{KellerLiverani99} and adapted to the random setting.
We point out that although their results may be
applied directly to some random perturbations of a single system, they yield
information about expectation of the random process only. In contrast,
ours yields information about almost all possible realizations.

Baladi and Gou\"ezel \cite{BaladiGouezel} introduced a relevant functional analytic setup for quasi-compactness of single transfer operators.
This followed other setups \cite{BlankKellerLiverani,GouezelLiverani,BaladiTsujii,DemersLiverani}.
In the present work we rely on the constructions from \cite{BaladiGouezel} because the results of \cite{GTQuas} allow one to show the existence of Oseledets splittings in that setting, a fact that is heavily exploited in our approach.

\subsection{Statement of the main results}
A random dynamical system consists of base dynamics (a measure-preserving map $\sigma$ of a probability space $\Omega$) and a family of
linear maps $\mathcal L_\omega$ from a Banach space $X$ to itself (in our applications these are the Perron-Frobenius operators of piecewise
expanding maps, $T_\omega$, of the circle). The results address stability of the dominant Oseledets subspace of the random dynamical system
when the linear maps are perturbed (leaving the base dynamics unchanged). We consider three classes of perturbation:
\begin{enumerate}[(A)]
\item \textit{Ulam-type perturbation.} For a fixed $k$, we define perturbed operators $\mathcal L_{k,\omega}$ to be $\mathbb E_k\circ \mathcal L_\omega$,
where $\mathbb E_k$ is the conditional expectation operator with respect to the partition into intervals of length $1/k$.

\item
 \textit{Convolution-type perturbation}. Given a family of densities $(Q_k)$ on the circle, we define perturbed operators $\mathcal L_{k,\omega}$ by
$\mathcal L_{k,\omega}f=Q_k * \mathcal L_\omega f$. If one applies $T_\omega$ and then adds a noise term with distribution given by $Q_k$,
then $\mathcal L_{k,\omega}$ is the random Perron-Frobenius operator. That is, the expectation of the Perron-Frobenius operators of
$\tau_y\circ T_\omega$ where $y$ has density $Q_k$ and $\tau_y$ is translation by $y$. Notable examples of perturbations of this type are the
cases where $Q_k$ is uniformly distributed on an interval $[-\epsilon_k,\epsilon_k]$ or where $Q_k$ is the $k$th Fej\'er kernel.

\item
\textit{Static perturbation}. Here one replaces the entire family of transformations $T_\omega$ by nearby transformations $T_{k,\omega}$. These are
much more delicate than the other two types of perturbation (composing with convolutions and conditional expectations generally make
operators more benign). Notice that by enlarging the probability space, perturbations of this type can include transformations with (for
example) independent identically distributed additive noise. To see this, let $\Xi$ denote the space of sequences taking values in $[-1,1]$,
equipped with the product of uniform measures and let $\bar\Omega=\Omega\times\Xi$ and $\bar\sigma$ be the product of $\sigma$ on the
$\Omega$ coordinate and the shift on the $\Xi$ coordinate. Then defining $T_{k,(\omega,\xi)}(x)=T_\omega(x)+\xi/k$ gives a family of perturbed
maps (with the common base dynamics being $\bar\Omega$). The unperturbed dynamics $(T_\omega)$ can, of course, also be seen as being driven
by $\bar\Omega$. Notice that this is \emph{not} the same thing as the perturbation obtained by convolving with a uniform $Q$. In the static
case, the results obtained would give a result that holds for compositions of $\mcl_{\omega,\xi}$ for almost every $\omega$ and almost every
sequence of perturbations $\xi$, whereas a result for the convolution perturbation would give a result that holds for the expectation of
these operators obtained by integrating over the $\xi$ variables. The convolution type perturbations are also known in the physics literature
as annealed systems, while the static perturbations are quenched systems.
\end{enumerate}

Below we outline the main application results of this paper. We refer the reader to \S\ref{sec:Ex} for
definitions, and to
 Theorems \ref{thm:Ulam}, \ref{thm:pertConv} and \ref{thm:NstepStab} for the precise statements.

\textbf{Theorem A:} \textit{(Stability under Ulam discretization).}
Let $\mcl$ be a random Lasota-Yorke map satisfying the conditions of \S\ref{S:LYMaps}.
Let $\{\mcl_k\}_{k\in \N}$ be the sequence of Ulam discretizations of $\mcl$, corresponding to uniform partitions of the domain into $k$ bins.
Assume $\mcl$ satisfies \textit{good} Lasota-Yorke inequalities (see \S\ref{Ssec:Ulam} for the precise meaning).
Then, for each sufficiently large $k$, $\mcl_k$ has a unique random acim.
Let $\{F_k\}_{k\in \N}$ be the sequence of random acims for $\mcl_k$. Then, $\lim_{k\to \infty}F_k=F$ fibrewise in $L_1$\footnote{\label{fn:conv} In fact, the fibrewise convergence holds in some fractional Sobolev norm $\hptt$, with $0<t'<\frac{1}{p}<1$, which in particular implies convergence in $L_p$ for some $p>1$. Since the domain is bounded, this yields convergence in $L_1$ as well.}.

\textbf{Theorem B:} \textit{(Stability under convolutions).}
Let $\mcl$ be a random Lasota-Yorke map satisfying the conditions of \S\ref{S:LYMaps}.
Assume $\mcl$ satisfies \textit{good} Lasota-Yorke inequalities (see \S\ref{Ssec:PertByConvolution} for the precise meaning).
Let $\{\mcl_k\}_{k\in \N}$ be a family of perturbations, arising from convolution with positive kernels $Q_k$, such that $\lim_{k\to \infty}\int Q_k(x)|x|\,dx=0$\footnote{This condition is equivalent to weak convergence of $Q_k$ to $\del_0$.}.
Then, for sufficiently large $k$, $\mcl_k$ has a unique random acim.
Let us call it $F_k$.
Then, $\lim_{k\to \infty}F_k=F$ fibrewise in $L_1\,^{\ref{fn:conv}}$.

\textbf{Theorem C:} \textit{(Stability under static perturbations).}
Let $\mcl$ be a random Lasota-Yorke map satisfying the conditions of \S\ref{S:LYMaps}.
Let $\{\mc{L}_k\}_{k\in \N}$ be a family of random Lasota-Yorke maps over the same base as $\mcl$, satisfying the conditions of \S\ref{S:LYMaps}, with the same bounds as $\mc{L}$.
Assume that there exists a sequence $\{\rho_k\}_{k>0}$ with $\lim_{k\to \infty}\rho_k= 0$ such that for $\paeom$, $d_{LY}(T_{k,\om}, T_\om)\leq \rho_k$, where $d_{LY}$ is a metric on the space of Lasota-Yorke maps. Furthermore, suppose that \emph{either}
\begin{enumerate}[(i)]
\item
$\{ T_\om\}_{\om \in \Om}$ satisfies a generalized \textit{no-periodic turning points} condition (see \S\ref{Ssec:DetPert} for full details); \emph{or}

\item
The expansion is sufficiently strong ($\mu^\ga>2$, where $\mu$ is a lower bound on $|DT_\om(x)|$, and $0<\ga\leq 1$ is the H\"older exponent of $DT_\om$), and $T_\omega$ depends continuously on $\omega$.
\end{enumerate}

Then, for every sufficiently large $k$, $\mcl_k$ has a unique random acim.
Let $\{F_k\}_{k\in \N}$ be the sequence of random acims for $\mcl_k$.
Then, $\lim_{k\to \infty}F_k=F$ fibrewise in $L_1\,^{\ref{fn:conv}}$.

\subsection{Structure of the paper}
The paper is organized as follows.
An abstract stability result, Theorem~\ref{thm:StabRandomAcim}, is presented in \S\ref{S:StabilityResult}, after introducing the underlying setup. Examples are provided in \S\ref{sec:Ex}. They include perturbations arising from finite-rank discretization schemes, perturbations by convolution, and  static perturbations  of random Lasota-Yorke maps. The theoretical results are illustrated with a numerical example in \S\ref{sec:numEx}.
Section~\ref{S:techPfs} contains proofs of the technical results.

\section{A stability Result}\label{S:StabilityResult}
\subsection{Preliminaries}
In this section, we introduce some notation.

\begin{defn}\label{defn:RandomDS}
A \textbf{strongly measurable separable random linear system with ergodic and invertible base}, or for short a \textbf{random dynamical system}, is a tuple $\mc{R}=(\Om,\mc{F}, \bbp, \sig, X, \mcl)$ such that
$(\Om, \mc{F}, \bbp)$ is a Lebesgue space, $\sig: (\Om, \mc{F}) \circlearrowleft$ is an invertible and ergodic $\bbp$-preserving transformation, $X$ is a separable Banach space, $L(X)$ denotes the set of bounded linear maps of $X$, and $\mcl: \Om \to L(X)$ is a strongly measurable map. We use the notation $\mcl_\om^{(n)}=\mcl(\sig^{n-1}\om) \circ \dots \circ \mcl(\om)$.
\end{defn}

\begin{defn}
  A random dynamical system $\mc{R}$ is called \textbf{quasi-compact} if $\ka^*(\mc{R})<\lam^*(\mc{R})$, where $\ka^*$, the \textbf{index of compactness}, and $\lam^*$, the \textbf{maximal Lyapunov exponent}, are defined as the following $\bbp$-almost everywhere constant limits:
\begin{align*}
  \lam^*&:=\lim_{n \to \infty} \frac{1}{n} \log \|\mcl_{\om}^{(n)}\|,\\
  \ka^*&:= \lim_{n\to \infty} \frac{1}{n} \log \|\mcl_\om^{(n)}\|_{ic(X)},
\end{align*}
where $\|A\|_{ic(X)}$ denotes the measure of non-com\-pact\-ness
\[
\|A\|_{ic(X)}:=\inf \{r>0 : A(B_X) \text{ can be covered by finitely
  many balls of radius } r \},
\]
and $B_X$ denotes the unit ball in $X$.

The \textbf{Lyapunov spectrum} at $\om\in \Om$ is the set $\Lam(\mc{R}(\om)):=\big\{ \lim_{n \to \infty} \frac{1}{n} \log \|\mcl_{\om}^{(n)}f\| : f\in X \big\}$.
A number $\lam$ is called an \textbf{exceptional Lyapunov exponent} if
$\lam \in \Lam(\mc{R}(\om))$ for $\paeom$ and $\lam>\ka^*$.
\end{defn}

\begin{defn}\label{def:temp}
A function $f:\Om \to \R$ is called
    \textbf{tempered with respect to $\sig$}, or simply \textbf{tempered} if $\sig$ is clear from the context, if for $\bbp$-almost every $\om$, $\lim_{n \to \pm \infty} \frac{1}{|n|} \log |f(\sig^n \om)|=0$.
\end{defn}

\begin{rmk}\label{rmk:temp}
It is straightforward to check that $f$ is tempered if and only if for every $\ep>0$, there exists a measurable function $g:\Om \to \R_+$ such that $f(\sig^n \om)\leq g(\om) e^{\ep|n|}$. Furthermore, $g$ may be chosen to be tempered. Indeed, one may replace $g$ by $\tld{g}(\om):=\inf_{n\in \Z} g(\sig^n\om)e^{\ep|n|}$. Also,
$\lim\sup_{n \to \pm \infty} \frac{1}{|n|} \log |\tld{g}(\sig^n \om)|\leq \ep$. 

A consequence of Tanny's theorem presented in \cite[Lemma C.3]{GTQuas}, states that either $\tld{g}$ is tempered or $\lim\sup_{n \to \pm \infty} \frac{1}{|n|} \log |\tld{g}(\sig^n \om)|=\infty$ for $\paeom$. Therefore, the previous paragraph shows that $\tld{g}$ is tempered.
\end{rmk}

We will rely on the following statement, which extends the work of Lian and Lu \cite{LianLu} by providing the existence of an \textit{Oseledets splitting} in the context of non-invertible operators.
\begin{thm}\cite[Theorem 2.10]{GTQuas}\label{thm:OselSpl}
Let $\mc{R}=(\Om,\mc{F}, \bbp, \sig, X, \mcl)$ be a quasi-compact strongly measurable separable linear random dynamical system with ergodic invertible base such that 
$\int \log^+\|\mcl\| \, d\bbp<\infty$.

Let $\lam^*=\lam_1>\dots > \lam_l>
  \ka^*$ be the (at most countably many) exceptional Lyapunov exponents of $\mc{R}$, and $m_1,
  \dots, m_l\in \N$ the corresponding multiplicities
  (or in the case $l=\infty$, $\lam_1>\lam_2>\ldots$ with $m_1,m_2,\ldots$ the
  multiplicities).

  Then, up to $\bbp$-null sets, there exists a unique, measurable,
  equivariant splitting of $X$ into closed subspaces, $X=V(\om)\oplus
  \bigoplus_{j=1}^l E_j(\om)$, where generally $V(\om)$ is infinite
  dimensional and $\dim E_j(\om)=m_j$.  Furthermore, for every $y\in
  E_j(\om)\setminus \{0\}$, $\lim_{n \to \infty} \frac{1}{n}\log
  \|\mcl_{\om}^{(n)}y\|=\lam_j$, for every $v\in V(\om)$, $ \limsup_{n
    \to \infty} \frac{1}{n}\log \|\mcl_{\om}^{(n)}v\|\leq \ka^*$ and the
  norms of the projections associated to the splitting are tempered
  with respect to $\sig$. 
\end{thm}
\begin{rmk}
The notation used for projections associated to the splitting are as follows.
  $\Pi_{j,\om}$ will denote the projection onto $E_j(\om)$ along $V(\om)\oplus
  \bigoplus_{i=1, i\neq j}^l E_i(\om)$; $\Pi_{\leq j,\om}$ will denote the projection onto $\bigoplus_{i=1}^{j} E_i(\om)$ along $V(\om)\oplus \bigoplus_{i= j+1}^l E_i(\om)$; $\Pi_{>j,\om}:=I - \Pi_{\leq j,\om}$.
\end{rmk}

\begin{defn}
A random dynamical system $\mc{R}=(\Om,\mc{F}, \bbp, \sig, X, \mcl)$ is called \textbf{splittable} if it has an Oseledets splitting.
Equivalently, one may also say that $\mc{R}$ \textbf{splits}, or \textbf{splits with respect to $X$} if the choice of Banach space needs to be emphasized.
\end{defn}

We conclude by making the following notational convention.
\begin{conv}
Throughout the paper, $C_\#$ denotes various constants that depend only on parameters $t, t', p, \ga$. The value of $C_\#$ may change from one appearance to the next.
\end{conv}

\subsection{Setting}\label{S:setting}
Let $(Y,|\cdot|)$, $(X,\|\cdot\|)$ be Banach spaces, with a compact embedding $X\hookrightarrow Y$, and a continuous embedding $Y\hookrightarrow L^1(Leb)$. 

Consider splittable random linear dynamical systems $\mc{R}=(\Om,\mc{F}, \bbp, \sig, X, \mcl)$ and $\mc{R}_k=(\Om,\mc{F}, \bbp, \sig, X, \mcl_k)$, $k\geq 1$, with a common ergodic invertible base $\sig:\Om \circlearrowleft$, and satisfying the following:
\begin{enumerate}
 \renewcommand{\theenumi}{H\arabic{enumi}} 
 \renewcommand{\labelenumi}{(\textbf{\theenumi})}
\setcounter{enumi}{-1} 
\item \label{it:TopExp}
$\int \log^+\|\mcl\| \, d\bbp<\infty$ and for every $k\in \N$, $\int \log^+\|\mcl_k\| \, d\bbp<\infty$.
Furthermore, for every $k\in \N$, $f\in X$ and $\paeom$,
$\mcl_{k,\om}$ and $\mcl_\om$ preserve the cone of non-negative functions, and satisfy
$\int \mcl_{k,\om}f\,dm=\int f \,dm=\int \mcl_\om f \,dm$.
\end{enumerate}
\begin{rmk}\label{rmk:H1}
In all the examples of this paper, condition \eqref{it:TopExp} is clearly satisfied. Thus, we will henceforth assume it holds and use it wherever needed.  
\end{rmk}

\begin{enumerate}
 \renewcommand{\theenumi}{H\arabic{enumi}} 
 \renewcommand{\labelenumi}{(\textbf{\theenumi})} 
\item \label{it:UnifLY} 
There exist a constant $B>0$ and a measurable $\al:\Om \to \R_+$ with
$\ka:=\int \log \al(\om) \, d\bbp(\om)<0$, such that for every $f\in X$, $k\in \N$ and $\paeom$,
\begin{equation}\label{eq:UnifLY}
 \|\mcl_{k,\om} f\|\leq \al(\om) \|f\|+B|f|.
\end{equation}
\end{enumerate}
\begin{rmk}\label{rmk:indComp} 
Whenever \eqref{it:UnifLY} holds, \cite[Lemma C.5]{GTQuas} ensures that $\ka^*(\mc{R}_k)<0$.
\end{rmk}

A version of the following statement was used by Buzzi \cite{Buzzi}, and is also derived in \cite[Lemma C.5]{GTQuas}:
Condition \eqref{it:UnifLY} is implied by the following more practical condition.
\begin{enumerate}
 \renewcommand{\theenumi}{H\arabic{enumi}'} 
 \renewcommand{\labelenumi}{(\textbf{\theenumi})} 
\item \label{it:GenUnifLY}
 $\{\log \|\mcl_{k,\om}\|\}_{k\in \N}$ is dominated by a $\bbp$-integrable function, and there exist measurable functions $\tld{\al}, \tld{B}:\Om \to \R_+$ with
$\int \log \tld{\al}(\om) \, d\bbp(\om)<0$, such that for every $f\in X$, $k\in \N$ and $\paeom$,
\begin{equation}\label{eq:GenUnifLY}
 \|\mcl_{k,\om} f\|\leq \tld{\al}(\om) \|f\|+ \tld{B}(\om)|f|.
\end{equation}
\end{enumerate}
Furthermore, $\ka$ in \eqref{eq:UnifLY} may be chosen arbitrarily close to $\int \log \tld{\al}(\om) \, d\bbp(\om)$.

\begin{enumerate}
 \renewcommand{\theenumi}{H\arabic{enumi}} 
 \renewcommand{\labelenumi}{(\textbf{\theenumi})}
\setcounter{enumi}{1} 
\item \label{it:PowBd}
For $\paeom$, $\sup_{k,n\in \N}|\mcl_{k,\sig^{-n}\om}^{(n)}|=:C(\om)$ is well defined and $C$ is tempered with respect to $\sig$.

\item \label{it:SmallPert} There exists $\tau_k\to0$ as $k\to \infty$ such that for $\paeom$,
\[
\sup_{\|g\|=1} |(\mcl_\om-\mcl_{k,\om})g|\leq \tau_k.
\]
\end{enumerate}

We conclude this section with the following lemma, which provides a temperedness condition analogous to that of \eqref{it:PowBd}, but with respect to the strong norm.
It will be used in bootstrapping arguments in the examples of \S\ref{sec:Ex}.
\begin{lem}\label{lem:PowerBoundStrongNorm}
Suppose conditions \eqref{it:UnifLY} and \eqref{it:PowBd} hold. Then,
$\sup_{k,n\in \N} \|\mcl_{k,\sig^{-n}\om}^{(n)}\|=:C'(\om)$ is well defined for $\paeom$, and $C'$ is tempered with respect to $\sig$.
In particular, $\lam_{k,1}\leq 0$ for every $k\in \N$.
Furthermore, if \eqref{it:TopExp} holds and if there exists $f\in X$ with $\int f\,dm\neq 0$, then for every $k\in \N$, $\lam_{k,1}=0$, where $\lam_{k,1}:=\lam^*(\mc{R}_k)$ is the maximal Lyapunov exponent of $\mcl_{k}$.
\end{lem}

\begin{proof}
Iterating \eqref{eq:UnifLY}, we get
\begin{align*}
 \|\mcl_{k,\sig^{-n}\om}^{(n)} f\| &\leq \al_{\sig^{-n}\om}^{(n)} \|f\| + B\sum_{j=0}^{n-1} \al_{\sig^{-n+j+1}\om}^{(n-j-1)} |\mcl_{k,\sig^{-n}\om}^{(j)} f|,
\end{align*}
where $\al_\om^{(l)}$ denotes the product $\al_\om \cdot \al_{\sig\om}\dots \al_{\sig^{l-1}\om}$ when $l>0$, and 1 when $l=0$.

Let $\ka=\int \log \al \, d\bbp<0$, and let $0<\ep<-\ka$. Then, there exists a tempered function $A$ such that for every $l\geq 0$ and $\paeom$, $\al_{\sig^{-l}\om}^{(l)}\leq A(\om) e^{(\ka+\ep/2)l}\leq A(\om)$. 

Since $C$ from \eqref{it:PowBd} is tempered, Remark~\ref{rmk:temp} shows there exists a tempered function $D:\Om \to \R_+$ such that $C(\sig^l\om)\leq D(\om) e^{\ep|l|/2}$.
Then,
\begin{align*}
 \|\mcl_{k,\sig^{-n}\om}^{(n)} f\| &\leq 
A(\om) \|f\| + B  \sum_{j=0}^{n-1} A(\om) e^{(n-j-1)(\ka+\ep/2)} C(\sig^{-n+j}\om)  |f|\\
&\leq A(\om)\Big( 1+ B D(\om) \sum_{j=0}^{n-1} e^\ep e^{(n-j-1)(\ka+\ep)} \Big) \|f\|\leq C'(\om)\|f\|,
\end{align*}
where $C'(\om):= A(\om) \Big( 1+ B D(\om)e^\ep/(1- e^{\ka+\ep}) \Big)$.
Since $A$ and $D$ are tempered, $C'$ is tempered, as claimed.

The second part of the lemma is immediate.
\qed \end{proof}

\subsection{Stability of random acims}
For each $n\in \N$ and $G:\Om \times I\to \R$, with $g_\om:=G(\om,\cdot)\in X$,
we let $\mcl^n G:\Om \times I\to \R$ be the function defined fibrewise by $\mcl^n G(\om,\cdot):=\mcl_{\sig^{-n}\om}^{(n)} g_{\sig^{-n}\om}$. Let $\mcl^n_k G$ be defined analogously.

Condition~\eqref{it:TopExp} shows that if $\mcl$ (or $\mcl_k$) has a non-negative fixed point, then one can in fact choose it to be fibrewise normalized in $L^1(\leb)$. In a slight abuse of notation, we call any such fixed point $F$ a \textit{random acim} for $\mcl$ (or $\mcl_k$), provided $\om \mapsto f_\om:=F(\om,\cdot)$ is $(\mc{F}, \mc{B}_X)$ measurable, where $\mc{B}_X$ is the Borel $\sig$-algebra of $(X, \|\cdot\|)$. 

The main result of this section is the following.
\begin{thm}\label{thm:StabRandomAcim}
Suppose $\mc{R}$ and $\mc{R}_k$, $k\geq 1$ are 
strongly measurable separable linear random dynamical systems with a common ergodic invertible
base satisfying conditions \eqref{it:TopExp}--\eqref{it:SmallPert}.
Assume $1\in X$\footnote{The conclusions remain valid if the condition $1\in X$ is replaced by the existence of a non-zero, non-negative element in $X$.}.

Assume that $\mc{R}$ splits (i.e. has an Oseledets splitting) with respect to $|\cdot|$
and $\dim E_1=1$. Then, $\mc{R}$ has a unique random acim, $F$.
For sufficiently large $k$, there is a unique random acim for $\mc{R}_k$, which is denoted by $F_k$.
Furthermore, $\lim_{k\to \infty}F_k=F$ fibrewise in $|\cdot|$. That is, for $\paeom$, $\lim_{k\to \infty}|f_\om-f_{k,\om}|=0$.
\end{thm}

\begin{proof}
 We divide the argument into three steps.
\begin{enumerate}
 \renewcommand{\theenumi}{\Roman{enumi}} 
 \renewcommand{\labelenumi}{\textit{(\theenumi)}.}
 
 \item 
\textit{If $\dim E_{1}=1$, then there is an attractive random acim.}\\
Recall that Lemma~\ref{lem:PowerBoundStrongNorm} ensures $\lam_{1}=0$.
Let $\lam_2<0$ be the second Lyapunov exponent of $\mc{R}$\footnote{In case 0 is the only exceptional Lyapunov exponent for $\mc{R}$, we let $\lam_2=\ka$, where $\ka$ is as in \eqref{it:UnifLY}.}.
Let $f_\om\in E_1(\om)$ be such that $\int f_\om \,dm=1$.
This normalization condition is possible, because
$X_0:=\{f\in X: \int f\,dm=0\}$ is the Oseledets complement to the top Oseledets space $E_{1}(\om)$. 
That is, $X_0=V(\om)\oplus \bigoplus_{j=2}^l E_j(\om)$ for $\paeom$. 
This follows from condition \eqref{it:TopExp}. In particular, if $g\in X$ is such that $\int g \,dm\neq 0$,
then $\int \Pi_{1,\om}(g)\,dm=\int g \,dm\neq 0$. Thus, $\Pi_{1,\sig^{-n}\om}(1)\neq 0$ for every $n\in \N$ and $\paeom$.

Let $g_n:=\Pi_{1, \sig^{-n}\om}(1)$, $h_n:=\Pi_{>1,\sig^{-n}\om}(1)$. 
\cite[Lemma 2.13(1)]{GTQuas} ensures that for each $\ep>0$ there is a measurable function $D'(\om)$ such that
\[
\|\mcl_{\sig^{-n}\om}^{(n)}h_n\|\leq D'(\om) e^{(\lam_2+\ep)n}\|h_n\|.
\]
Temperedness of $\Pi_{>1}$, coming from Theorem~\ref{thm:OselSpl}, shows there is a measurable function $\tld{D}(\om)$ such that $\|h_n\|\leq \tld{D}(\om) e^{\ep n}$. Combining, with the above, one gets
\[
\|\mcl_{\sig^{-n}\om}^{(n)}h_n\|\leq D'(\om)\tld{D}(\om) e^{(\lam_2+2\ep)n}.
\]
By linearity of $\mcl^{(n)}_{\sig^{-n}\om}$, $\lim_{n\to \infty} d_X \big( \mcl^{(n)}_{\sig^{-n}\om} 1,E_1(\om) \big)=0$, where $d_X$ denotes the distance with respect to the norm on $X$, $\|\cdot\|$.
Also, the normalization condition on $f_\om$ ensures that
$\lim_{n\to \infty} \mcl^{(n)}_{\sig^{-n}\om} 1= f_\om $ in $X$, for $\paeom$. In particular, $f_\om$ is non-negative, as $X$ is continuously embedded in $L^1$ by assumption.
Thus $F=\lim_{n\to \infty} \mcl^n 1$ (fibrewise limit in $X$) is a random acim for $\mcl$. Measurability follows from measurability of $E_1(\om)$, guaranteed by Theorem~\ref{thm:OselSpl}.

\item
\textit{For sufficiently large $k$, $\mc{R}_k$ has a unique random acim.}\\
Lemma~\ref{lem:PowerBoundStrongNorm} yields $\lam_{k,1}=0$ for every $k$.
Thus, Remark~\ref{rmk:indComp} ensures $\mc{R}_k$ is quasi-compact and Theorem~\ref{thm:OselSpl}
shows that $\mc{R}_k$ splits with respect to $X$ for every $k\in \N$.

Let $f\in X$. Then,
\[
 \mcl^{(n)}_\om f -\mcl_{k,\om}^{(n)}f=
\sum_{j=0}^{n-1}\mcl^{(j)}_{\sig^{n-j}\om}(\mcl_{\sig^{n-j-1}\om}-\mcl_{k,\sig^{n-j-1}\om}) \mcl_{k,\om}^{(n-j-1)}f.
\]
Let $g_{k,\om, i}:=(\mcl_{\sig^{i}\om}-\mcl_{k,\sig^{i}\om}) \mcl_{k,\om}^{(i)}f$.
Then, since $\int g_{k,\om,i}(x)dx=0$, we have that $g_{k,\om, i}\in E_{>1}(\om')$ for 
$\bbp\text{-a.e. } \om' \in \Om$,
where $E_{>1}(\om):=V(\om) \oplus \bigoplus_{j=2}^l E_j(\om)$ is the complementary Oseledets space to $E_1(\om)$.
Thus, \cite[Lemma 2.13(1)]{GTQuas} ensures that for each $\ep>0$ there is a measurable function $D'(\om)$ such that 
\begin{equation}\label{eq:uniq1}
 \begin{split}
 |\mcl^{(n)}_\om f -\mcl_{k,\om}^{(n)}f| &\leq \sum_{j=0}^{n-1} D'(\sig^n \om)e^{j(\lam_2+\ep)}
|(\mcl_{\sig^{n-j-1}\om}-\mcl_{k,\sig^{n-j-1}\om}) \mcl_{k,\om}^{(n-j-1)}f| \\
&\leq
\sum_{j=0}^{n-1} D'(\sig^n \om)e^{j(\lam_2+\ep)} \tau_k \| \mcl_{k,\om}^{(n-j-1)}f\|.
\end{split}
\end{equation}
We let $A, C'$ be the tempered functions from the proof of Lemma~\ref{lem:PowerBoundStrongNorm}.
Thus, we have
\begin{align*}
\| \mcl_{k,\om}^{(n-j-1)}f\|\leq \al_{\om}^{(n-j-1)}\|f\|+C'(\sig^{n-j-1}\om)|f|.
\end{align*}
Substituting into \eqref{eq:uniq1}, we get
\begin{equation}\label{eq:uniq2}
 \begin{split}
 |\mcl^{(n)}_\om f -\mcl_{k,\om}^{(n)}f| &\leq 
D'(\sig^n \om) \tau_k \sum_{j=0}^{n-1} e^{j(\lam_2+\ep)} \Big( \al_{\om}^{(n-j-1)}\|f\|+C'(\sig^{n-j-1}\om)|f| \Big)\\
&\leq D'(\sig^n \om) \tau_k \sum_{j=1}^{n} e^{j(\lam_2+\ep)}
\Big( A(\om)e^{(\ka+\ep)(n-j)}\|f\|+\tld{C}(\sig^{n}\om)e^{2\ep j} |f| \Big),
\end{split}
\end{equation}
where $\tld{C}$ is such that for every $l\in \Z$, $C'(\sig^l \om)\leq \tld{C}(\om) e^{2\ep|l|}$.

Now, since $\mc{R}_k$ splits, $\dim E_{k,1}(\om)<\infty$. In particular, there exists a measurable function $J_k(\om)>0$ such that
$\sup_{f\in E_{k,1}(\om)\setminus \{0\}} \frac{\|f\|}{|f|}\leq J_k(\om)$.
Let $M_{J_k}$ be such that $\bbp(\{\om: J_k(\om)<M_{J_k}\})>0.9$, and let $G_{J_k}=\{\om: J_k(\om)<M_{J_k}\}$.
Thus, if $\om \in G_{J_k}$,  and $f\in E_{k,1}(\om)$,
\begin{align}\label{eq:contWeakNorm3}
 |\mcl^{(n)}_\om f -\mcl_{k,\om}^{(n)}f| &\leq \tau_k D''(\sig^n \om) A'(\om) \big(n e^{(\lam_2+\ep)n} M_{J_k}+ 1 \big)|f|,
\end{align}
where $D''= D' \max (1, \tld{C}),$ and $A'=\max(A,\frac{1}{1-e^{\lam_2+3\ep}})$, where we assume $\ep$ is such that $\lam_2+3\ep<0$.

Let $M_A, M_{D}$ be such that $\bbp(\{\om: A'(\om)<M_A\})>0.9$ and $\bbp(\{\om: D''(\om)<M_D\})>0.9$. Let $G_A=\{\om: A'(\om)<M_A\}$, $G_D=\{\om: D''(\om)<M_D\}$.
Thus, if $k$ is sufficiently large, there exists $N=N_k$ such that if $n\geq N$, $\om \in G_A\cap G_{J_k}$\footnote{Let us observe that the intersection is non-empty by the choice of $M_{J_k}$ and $M_A$.}, $\sig^n \om \in G_D$, and $f\in E_{k,1}(\om)$, we have
\begin{equation}\label{eq:contWeakNorm1}
|\mcl^{(n)}_\om f -\mcl_{k,\om}^{(n)}f| \leq \tfrac{1}{4}|f|.
\end{equation}

Now suppose also $\int f \,dm=0$.
Then, the assumptions on $\mc{R}$ show that there exists $N'=N'(\om)$ such that for every $n\geq N'$, 
$|\mcl^{(n)}_\om f|<\tfrac{1}{4}|f|$.
Combining with \eqref{eq:contWeakNorm1}, we get that there exists $\tld{N}=\tld{N}_k(\om)$ such that if $n\geq \tld{N}$,
$\om \in G_A\cap G_{J_k}$, $\sig^n \om \in G_D$, then 
\begin{equation}\label{eq:contWeakNorm2}
|\mcl_{k,\om}^{(n)}f| \leq \tfrac{1}{2}|f|,
\end{equation}
for every $f\in E_{k,1}(\om)$ with $\int f \,dm=0$.

Now, suppose for a contradiction that $d_{k,1}:=\dim E_{k,1}>1$ (recall that $\dim E_{k,1}$ is $\bbp$ almost everywhere constant). 
We can find a constant $M_{\tld{N}_k}$ such that the set
$$G_k:=\big\{\om\in \Om: J_k(\om)<M_{J_k},\ A'(\om)< M_A,\ D''(\om)< M_D,\ \tld{N}_k(\om)< M_{\tld{N}_k},\ \dim E_{k,1}(\om)=d_{k,1} \big\}$$ has positive $\bbp$-measure.

Then, by Birkhoff's ergodic theorem, for each $\ep>0$ there is a subset $G'_k\subset G_k$ of full $\bbp$-measure in $G_k$, such that for each $\om \in G'_k$, there exists $N_0=N_0(\om)$ such that for every $n\geq N_0$,
there exists a sequence $0= n_0<n_1<n_2<\dots n_l \leq n$, with $l\geq \frac{n(1-\ep)\bbp(G_k)}{\tld{N}_k}$, such that
(i) for every $0\leq j \leq l$, $\sig^{n_j}\om \in G_k$, and 
(ii) for every $0\leq j < l$, $n_{j+1}-n_{j}\geq \tld{N}_k$.
The $\tld{N}_k$ in the denominator comes from condition (ii), as we may have to choose one out of each $\tld{N}_k$ visits to $G_k$ to ensure this, in the worst-case scenario.

Let $\om \in G_k'$, and $f\in E_{k,1}(\om)$ be such that $\int f \,dm=0$ and $|f|=1$.
Using \eqref{eq:contWeakNorm2} at times $n_1, \dots, n_l$, we get that 
$|\mcl_{k,\om}^{(n_l)}f| \leq \frac{1}{2^l}$,
so $\frac{1}{n_l}\log |\mcl_{k,\om}^{(n_l)}f| \leq - \frac{l}{n_l} \log 2 \leq - \frac{(1-\ep)\bbp(G'_k)}{\tld{N}_k} \log 2$. Thus, $\liminf_{n\to \infty}\frac{1}{n}\log |\mcl_{k,\om}^{(n)}f|<0$.
Since $f\in E_{k,1}(\om)$, this contradicts the fact that $\lam_{k,1}=0$, because
\cite[Theorem 3.3]{FroylandStancevic} shows that $\bbp$-almost surely, $\liminf_{n\to \infty}\frac{1}{n}\log |\mcl_{k,\om}^{(n)}f|=\lam_{k,1}$. That is, the Lyapunov exponents of $f$ with respect to strong and weak norms coincide.

Therefore, $\dim(E_{k,1})=1$ for every sufficiently large $k$.
Existence and uniqueness of a random acim for $\mc{R}_k$ follow from the previous step.

\item
\textit{Fibrewise convergence of random acims.}\\
For each sufficiently large $k$, let $F_k$ be the unique random acim guaranteed by the previous step.
Let
\begin{align*}
R_{k,n, \om}:=(\mcl^{(n)}_{\sig^{-n} \om} - \mcl^{(n)}_{k,\sig^{-n} \om}) 1.
\end{align*}
We keep the notation as above.
Starting from Equation~\eqref{eq:uniq2} and arguing as in the previous step, we get
\begin{align*}
|R_{k,n,\om}|&\leq \tau_k D''(\om) A'(\sig^{-n}\om) K \|1\|,
\end{align*}
where $K:=\max_{n\in \N} \big(n e^{(\lam_2+\ep)n} + 1 \big)$.

Thus, whenever $\om, \sig^{-n} \om \in G:=G_A\cap G_D$,
\begin{align}\label{eq:boundRem}
|R_{k,n,\om}|&\leq \tau_k M_D M_A K \|1\|.
\end{align}
Since $\bbp(G)>0$, we may choose $G'\subset G$ with $\bbp(G')>0$ such that for every $\om \in G'$, there exists an infinite sequence $(n_j)$ such that $\sig^{-n_j}\om \in G$, by the Poincar\'e recurrence theorem.
Recalling that $\lim_{k\to\infty}\tau_k=0$, \eqref{eq:boundRem} shows that for every $\om \in G'$,
 \begin{align*}
\lim_{k\to \infty}|R_{k,n_j,\om}|= 0, \text{uniformly in }j.
\end{align*}

Furthermore, as we showed in steps (I) and (II), $\mathcal
L^{(n)}_{\sigma^{-n}\omega}1$ converges to $f_\omega$ as $n\to\infty$ and for
sufficiently large $k$, $\mathcal L^{(n)}_{k,\sigma^{-n}\omega}1$ converges to
$f_{k,\omega}$.
Hence, all the following limits exist (in $|\cdot|$) and 
 \begin{align*}
f_\om-f_{k,\om}=\lim_{n\to \infty} \mcl_{\sig^{-n}\om}^{(n)} 1 - \mcl_{k,\sig^{-n}\om}^{(n)} 1=\lim_{n\to \infty} R_{k,n,\om}.
\end{align*}
Therefore, for every  $\om \in G'$, $\lim_{k\to \infty}|f_\om-f_{k,\om}|=\lim_{k\to \infty} \lim_{j\to \infty} |R_{k,n_j,\om}|=0$.
Since the set of $\om$ for which $\lim_{k\to \infty}f_{k,\om}= f_\om$ is $\sig$-invariant, ergodicity of $\sig$ yields fibrewise convergence for $\paeom$, as claimed.\qed
\end{enumerate}
\end{proof}

\section{Examples}\label{sec:Ex}
The first subsection, \S\ref{S:LYMaps}, is devoted to introducing the setting of random Lasota-Yorke maps.
As the hypotheses of the stability theorem of \S\ref{S:StabilityResult} do not hold in the classical setting of functions of bounded variation, the alternative norms used, as well as some relevant properties thereof, are presented.

Three applications of the stability theorem in the context of random Lasota-Yorke maps are presented in sections \S\ref{Ssec:Ulam}--\S\ref{Ssec:DetPert}. These correspond to Ulam approximations, perturbations with additional randomness that arise by taking convolution with non-negative kernels, and static perturbations, 
respectively. 
In \S\ref{sec:numEx}, we illustrate the results with a numerical example.

\subsection{Setting: Random (non-autonomous) Lasota-Yorke maps}\label{S:LYMaps}
For $0<\ga\leq 1$, let $\LY^\ga$ be the space of finite-branched piecewise $C^{1+\ga}$ expanding maps of the interval.
For each $T\in \LY^\ga$, let $\nint^T$ is the number of branches of $T$, and $\{ I_i^T\}_{1\leq i \leq \nint^T }$ the continuity intervals of $T$ and $DT$. 
We recall the definition of the metric $d_{\LY}$ on the space of random Lasota-Yorke maps $\LY^{1+\ga}$ used in \cite{GTQuas}
\footnote{For convenience, the $\ga$ dependence of the distance does not explicitly appear in the notation. We think of $\ga$ as fixed.}. 
Let $S,T\in\LY^{1+\ga}$. Let the branches for $S$ be
$(I^S_i)_{i=1}^{\nint^S}$ and for $T$ be $(I^T_i)_{i=1}^{\nint^T}$, considered
as an \emph{ordered} collection. If $\nint^T\ne \nint^S$,
or $I^S_i\cap I^T_i=\emptyset$ for some $i$, we define $d_{\LY}(S,T)=1$.
Otherwise we define
\[
d_{\LY}(S,T)=
   \max_i \|(S_i-T_i)\vert_{I_i^S\cap I_i^T}\|_{C^{1+\ga}} + \max_i\Big | \|S_i \|_{C^{1+\ga}}-  \|T_i \|_{C^{1+\ga}} \Big| +\max_i
   d_H(I_i^S, I_i^T),
\]
where $d_H$ denotes Hausdorff distance. We endow $\LY^{1+\ga}$ with the corresponding Borel
$\sigma$-algebra $\mc{B}_{LY}$.

\begin{defn}
Let $(\Om,\mc{F})$ be a measurable space.
A \textbf{random Lasota-Yorke map} $\mc{T}$ is a measurable function $\mc{T}: (\Om, \mc{F}) \to (\LY^\ga, \mc{B}_{LY})$.
\end{defn}
\begin{rmk}
A random Lasota-Yorke map can be made into a random dynamical system by fixing a separable Banach space $X$ on which transfer operators $\mcl_{\om}$ associated to $\mc{T}(\om)=:T_\om$ are bounded linear maps. 
\end{rmk}

Let $\mc{T}$ be a random Lasota-Yorke map.
Also, assume there exist uniform bounds $\nint, \dist, \mu$ as follows
\begin{enumerate}
 \renewcommand{\theenumi}{M\arabic{enumi}} 
 \renewcommand{\labelenumi}{(\theenumi)}
\item \label{it:M1} $\nint_\om\leq \nint$, where $\nint_\om$ is the cardinality of the partition of $\dom=\{I_{1,\om}, \dots, I_{\nint_\om, \om}\}$ into domains of differentiability of $T_\om$,
\item \label{it:M2} $\|T_{i,\om}^{\pm 1}\|_{C^{1+\ga}}\leq \dist$,
\item \label{it:M3} $\essinf_{\om\in\Om,x\in \dom}|(DT_\om)_x|>\mu>1$.
\end{enumerate}
\begin{rmk}\label{rmk:UnifBdNorm}
It follows from \cite[\S4]{BaladiGouezel} (see also \cite{Thomine} or \cite[\S3]{GTQuas}) that 
conditions \eqref{it:M1}--\eqref{it:M3} yield $\esssup_{\om\in  \Om} \|\mcl_\om\|<\infty$.
\end{rmk}

Furthermore, assume $\mc{T}$ enjoys Buzzi's \textit{random covering condition} \cite{BuzziEDC}.
That is, assume that for every non-trivial interval $J\subset I$ and $\paeom$, there exists some $n\in \N$ such that $T_\om^{(n)}(J)=I \pmod{0}$.
This ensures $\dim(E_1)=1$ \cite{BuzziEDC}, as required by Theorem~\ref{thm:StabRandomAcim}.
Let us call $F$ the unique random acim of $\mc{T}$.

Next, let us observe that conditions \eqref{it:M1}--\eqref{it:M3} above yield, for each $N\in \N$, a Lasota-Yorke inequality 
for the norms $BV, L^1$ of the form:
\begin{equation}\label{eq:L1BV-LY-N}
 \|\mcl_{\om}^{(N)} f\|_{BV}\leq \al_N(\om) \|f\|_{BV}+B_N(\om)\|f\|_1,
\end{equation}
for $\paeom$.
Furthermore, the work of Rychlik \cite{Rychlik} shows that if $N$ is sufficiently large (depending on the constants from \eqref{it:M1}--\eqref{it:M3}), one can ensure $\int \log \al_N(\om)\, d\bbp<0$.
We will assume the above holds for $N=1$, and discuss the necessary modifications needed to cover the case $N>1$ separately in each subsection. That is, we assume for $\paeom$,
\begin{equation}\label{eq:L1BV-LY}
 \|\mcl_{\om} f\|_{BV}\leq \al(\om) \|f\|_{BV}+B(\om)\|f\|_1,
\end{equation}
with $\int \log \al(\om)\, d\bbp<0$.

Also, since $\mcl_{\om}$ preserves the non-negative cone and $\int f =\int \mcl_\om f \,dm$, then
$\|\mcl_{\om}^{(n)}\|_1\leq 1$ for $\paeom$ and $n\in \N$.
Thus, Lemma~\ref{lem:PowerBoundStrongNorm} ensures $\sup_{n\in \N}\|\mcl_{\sig^{-n}\om}^{(n)}\|_{BV}$ is tempered with respect to $\sig$.
Further,
\begin{equation}\label{eq:LinftyBd}
  \|\mcl_{\sig^{-n}\om}^{(n)} f\|_\infty\leq \|\mcl_{\sig^{-n}\om}^{(n)} 1\|_\infty \|f\|_\infty\leq \|\mcl_{\sig^{-n}\om}^{(n)} 1\|_{BV} \|f\|_\infty,
\end{equation}
which ensures $\sup_{n\in \N}\|\mcl_{\sig^{-n}\om}^{(n)}\|_{\infty}$ is tempered as well.
 The Riesz-Thorin interpolation theorem (see e.g. \cite[VI.10.12]{DunfordSchwartz}) then yields temperedness of $\sup_{n\in \N}\|\mcl_{\sig^{-n}\om}^{(n)}\|_{p}$ for every $1\leq p \leq \infty$.

It follows from \cite[\S3]{GTQuas}, which in turn builds on \cite{BaladiGouezel}, that there exist
$t,p\in \R$, with 
\[
0<t<\tfrac{1}{p}<1,
\] 
such that $\mc{T}$ splits with respect to $\hpt$, where $\hpt$ is the subset of the fractional Sobolev space with parameters $p$ and $t$, whose support lies inside the interval $I$. 
That is, 
\[
\hpt= \mc{F}^{-1}(m_{-t} \mc{F}(L_p)) \cap \{f\in L_p: \spt(f)\subset I\},
\]
where $\mc{F}$ denotes the Fourier transform, and $m_t(\zeta):=
(1+|\zeta|^2)^{\frac{t}{2}}$.
The norm on $\hpt$ is given by
\[
\|f\|_{\hpt}:=\|\mc{F}^{-1}(m_t \mc{F}(f))\|_{p}.
\]
The specific choice of parameters $p$ and $t$ depends on the random map. 
In general, $p>1$ may need to be chosen arbitrarily close to 1. Also, $0<t<\frac{1}{p}$ is a smoothness parameter, and it can not be larger than the smoothness of the (derivative of the) maps. 

It was also shown in \cite[\S3]{GTQuas} that there are
constants $\tld{N} \in \N, \tld{\al}_{\tld{N}}(\om), \tld{B}_{\tld{N}}(\om)>0$ such that for every $f\in \hpt$, and $\paeom$,
\begin{equation}\label{eq:LY-hpt-lp}
 \|\mcl_{\om}^{\tld{N}}f\|_{\hpt} \leq \tld{\al}_{\tld{N}}(\om) \|f\|_{\hpt}+\tld{B}_{\tld{N}}(\om)\|f\|_p,
\end{equation}
with $\int \log\tld{\al}_{\tld{N}}(\om)\, d\bbp<0$.
From temperedness of  $\sup_{n\in \N}\|\mcl_{\sig^{-n}\om}^{(n)}\|_{p}$, a second application of Lemma~\ref{lem:PowerBoundStrongNorm} ensures $\sup_{n\in \N}\|\mcl_{\sig^{-n}\om}^{(n)}\|_{\hpt}$ is tempered with respect to $\sig$.

Finally, we fix our attention on yet another pair of norms, which will be the pair for which we use the results of \S\ref{S:StabilityResult}.
This choice is motivated by the necessity of a splitting for $\mc{T}$ with respect to the weak norm, coming from Theorem~\ref{thm:StabRandomAcim}.
Let 
$t$ and $p$ be as above, and let
\[
0<t'<t,
\] 
be sufficiently close to $t$ so that $\mc{T}$ splits with respect to $\hptt$ (by
the same argument as above).

The $\hpt$ norm is stronger than $\hptt$, and the embedding $\hpt\hookrightarrow \hptt$ is compact.
Also, the $\hptt$ norm is stronger than $L_p$, so \eqref{eq:LY-hpt-lp} implies
\[
 \|\mcl_{\om}^{\tld{N}}f\|_{\hpt} \leq \tld{\al}_{\tld{N}}(\om) \|f\|_{\hpt}+\tld{B}_{\tld{N}}(\om)
\|f\|_{\hptt},
\]
for every $f\in X$, and $\paeom$, with $\int \log\tld{\al}_{\tld{N}}(\om)\, d\bbp<0$.

As above, we will assume the above holds for $\tld{N}=1$,
and discuss the case $\tld{N}>1$ separately in each section.

We conclude this section with the following remark, 
which will allow us to work interchangeably on either $(\hpt, \hptt)$, or the corresponding fractional Sobolev spaces on the circle, $(\hptorus, {H}_p^{t'}(\T))$.
\begin{rmk}\label{rmk:EquivHptNorms}
Let $\hptorus$ be the fractional Sobolev space on the circle. That is, 
\begin{equation}\label{def:hptorus}
 \hptorus=\big\{ f \in D'(\T): \|f\|_{\hptorus}:= \big\|\sum_{j\in \Z}\langle j\rangle^{t/2} \hat{f}(j) e^{2\pi ijx} \big\|_p  <\infty \big\},
\end{equation}
where $\langle j \rangle := 1+j^2$ and $D'(\T)$ is the dual of the space of $C^\infty$ functions on $\T$.

Then, if $0<t<\frac{1}{p}$, $\|\cdot\|_{\hptorus}$
and $\|\cdot\|_{\hpt}$ are equivalent norms \cite[\S4.3.2 \& \S4.11.1]{TriebelInterpolation}.
\end{rmk}

\subsubsection{Approximation by smooth functions}\label{S:approxSmoothFunctions}
Let $f\in \hpt$ and $\ep>0$. Set
\begin{equation}\label{eq:fep}
 f_\ep:=\sum_{j\in \Z} e^{-\ep (1+(2\pi j)^2)} a_j \phi_j(x),
\end{equation}
where $\phi_j(x)=e^{2\pi i j x}$ and $a_j:=\int_{0}^1 \phi_j(y)f(y)\, dy$.
Then, $f_\ep \in C^\infty$. Since $|a_j|\leq \|f\|_1$, a direct calculation shows
\begin{equation}\label{eq:BdFepC2}
 \|f_\ep\|_{C^2}\leq \sum_{j=1}^\infty e^{-\ep (1+(2\pi j)^2)} (2\pi j)^2\|f\|_1\leq C(\ep)\|f\|_{\hpt},
\end{equation}
for some decreasing function $C:(0,\infty)\to \R_+$ such that $\lim_{\ep\to 0^+} C(\ep)=\infty$.

\begin{conv}
For the remainder of the section, whenever $\|\cdot\|$ and  $|\cdot|$ appear, they stand for  $\|\cdot\|_{\hpt}$ and $\|\cdot\|_{\hptt}$, respectively. In particular, whenever the hypotheses of Theorem~\ref{thm:StabRandomAcim} are verified, it is meant with respect to this pair of norms.
\end{conv}

We will make use of the following approximation lemma, whose proof is deferred until \S\ref{pf:regApprox}.
\begin{lem}\label{lem:regApprox}
Let $f\in \hpt$. Then,
$|f_\ep-f| \leq C_\# \ep^{\frac{t-t'}{2}} \|f\|$.
\end{lem}

\subsection{The Ulam scheme}\label{Ssec:Ulam}
For each $k\in \N$, let  $\mc{P}_k=\{B_1, \dots, B_k\}$ be a partition of $\dom$ into $k$ subintervals of uniform length, called bins.
Let $\E_k$ be given by the formula 
\[
\E_k(f)=\sum_{j=1}^k \frac{1}{m(B_j)}\Big(\int 1_{B_j}\,f\,\,dm\, \Big)1_{B_j},
\]
where $m$ denotes normalized Lebesgue measure on $\dom$.

Let $\mcl_k$ be defined as follows.
For each $\om \in \Om$, $\mcl_{k,\om}:=\E_k \mcl_\om$.
This is the well-known Ulam discretization \cite{Ulam}, in the context of non-autonomous systems. It provides a way of approximating the transfer operator $\mcl$ with a sequence of (fibrewise) finite-rank operators $\mcl_k$, each
taking values on functions that are constant on each bin.

It is well known that $\E_k$ is a contraction in $L_p$. For $\hpt$ norms, we are not aware of any similar results in the literature. Here we establish the following, which may be of independent interest. 
\begin{lem}\label{thm:hptbound}
Let $p>1$ and let $0<t<1/p$. There exists a constant $C_\#$
such that $\|\E_k f\| \le C_\#\|f\|$ for all $k\ge 1$ and all $f\in \hpt$\footnote{We remind the reader that $C_\#$ may depend on parameters $p,t$, and that $\|\cdot\|$ denotes $\|\cdot\|_{\hpt}$ throughout this section.}.
\end{lem}
The proof of Lemma~\ref{thm:hptbound} is deferred until \S\ref{S:pfBddEnHpt}.
We can therefore obtain a uniform inequality like \eqref{eq:UnifLY}, provided the expansion of $T_\om$ is sufficiently strong. 

\begin{thm}\label{thm:Ulam}
Let $\mcl$ be a non-autonomous Lasota-Yorke map, as defined in \S\ref{S:LYMaps}.
For each $k\in \N$, let $\mcl_k$ be the sequence of Ulam discretizations, corresponding to the partition $\mc{P}_k=\{B_1, \dots, B_k\}$ introduced above.

Assume $\mcl$ satisfies the Lasota-Yorke inequalities \eqref{eq:L1BV-LY} and \eqref{eq:GenUnifLY} with 
$\al$ and $\tld{\al}$ such that $\int \log \al(\om)\, d\bbp<0$ and $\log C_\# + \int \log \tld{\al}(\om)\, d\bbp<0$, where $C_\#=\sup_{k\in \N} |\E_k|$, guaranteed to be finite by Lemma~\ref{thm:hptbound}\footnote{If the contraction condition of Theorem~\ref{thm:Ulam} requires taking either $N>1$  in \eqref{eq:L1BV-LY-N} and/or $\tld{N}>1$ in \eqref{eq:LY-hpt-lp}, the conclusions remain valid provided the projection $\E_k$ is taken after $\max\{N, \tld{N}\}$ compositions.}.

Then, for each sufficiently large $k$, $\mcl_k$ has a unique random acim.
Let $\{F_k\}_{k\in \N}$ be the sequence of random acims for $\mcl_k$. Then, $\lim_{k\to \infty}F_k=F$ fibrewise in $|\cdot|$.
\end{thm}

\begin{proof}
We will verify the assumptions of Theorem~\ref{thm:StabRandomAcim}. \eqref{it:TopExp} is immediate.
The assumptions combined with Lemma~\ref{thm:hptbound} ensure that \eqref{it:GenUnifLY} holds as well.
Condition~\eqref{it:PowBd} follows exactly as in the bootstrapping argument in \S\ref{S:LYMaps}.

The last condition to check in order for Theorem~\ref{thm:StabRandomAcim} to apply is \eqref{it:SmallPert}, which follows from the next proposition.
\begin{prop}\label{lem:smallPertUlam}
There exists a sequence $\{\tau_k\}_{k>0}$ with $\lim_{k\to \infty}\tau_k= 0$ such that
\[
\sup_{\|g\|=1} |(\mcl_\om-\mcl_{k,\om})g|\leq \tau_k.
\]
\end{prop}

\begin{proof}[Proof of Proposition~\ref{lem:smallPertUlam}]
Let $\eta_k= m(B_j)=\frac{1}{k}$ be the  diameter of the partition elements of $\mc{P}_k$.
Let $f\in \hpt$, and for each $\ep>0$, let $f_\ep$ as in \eqref{eq:fep}. For each $k\in \N$, we have
\begin{align}\label{eq:UlamError}
  |(\mcl_{k,\om}-\mcl_\om)f| &\leq  |(\mcl_{k,\om}-\mcl_\om)f_\ep| + |(\mcl_{k,\om}-\mcl_\om)(f_\ep-f)|=: (U1) + (U2).
\end{align}
We will bound each term separately.
The fact that 
\begin{equation}\label{eq:U2}
(U2)\leq C_\#\ep^{\frac{t-t'}{2}}\|f\|
\end{equation}
follows from Lemma~\ref{lem:regApprox} and Remark~\ref{rmk:UnifBdNorm}, after recalling that $|\E_k|$ is bounded independently of $k$, by Lemma~\ref{thm:hptbound}. 

Now we estimate $(U1)$. Let $\big\{ I_{i,\om} \}_{1\leq i \leq \nint_\om}$ be the partition of $I$ into domains of differentiability of $T_\om$, $Q_{i,\om}=T_\om (I_{i,\om})$, and $\xi_{i,\om}:=\big( T_\om|_{I_{i,\om}} \big)^{-1}$.
Then, the transfer operator $\mcl_\om$ is given by the following expression, 
\[
\mcl_\om f=\sum_{i=1}^{\nint_\om}  1_{Q_{i,\om}} \cdot f \circ \xi_{i,\om} \cdot |D\xi_{i,\om}|.
\]
This is the sum of at most $\nint_\om$ terms of the form $1_J g$, where $J\subset I$ is an interval, and $g=f \circ \xi_{i,\om} \cdot |D\xi_{i,\om}|$ for some $i\in \N, \om\in \Om$. Furthermore, when $f\in C^\ga$, each such $g$ is also $C^\ga$, and $\|g\|_{C^\ga}\leq C\|f\|_{C^\ga}$, where $C$ depends on the random map $\mc{T}$, but not on $f$. 
In this case, $\mcl_\om f$ may be rewritten as 
\[
 \mcl_\om f= (\mcl_\om f)^h + (\mcl_\om f)^s,
\]
where $(\mcl_\om f)^h\in C^\ga$ is such that $\|(\mcl_\om f)^h\|\leq C\|f\|_{C^\ga}$,
and $(\mcl_\om f)^s$ is the sum of at most $2\nint_\om$ step functions, with jumps of size at most $\nint_\om\|f\|_\infty$. 
Then, $\|(\mcl_\om f)^s\|_\infty\leq 2\nint_\om^2\|f\|_\infty$ and 
\begin{align*}
 (U1)= |(\E_k -I)\mcl_\om f_\ep|\leq |(\E_k -I)(\mcl_\om f_\ep)^h|+|(\E_k -I)(\mcl_\om f_\ep)^s|=: (U11)+(U12).
\end{align*}

To estimate the first term, we rely on the following lemma, whose proof is deferred until \S\ref{pf:UlamError4Holder}.
\begin{lem}\label{lem:UlamError4Holder}
Let $g\in C^\ga$, with $t'<\min\{\ga, \frac{1}{p}\}$. Then, $|(\E_k -I)g|\leq C_\#\|g\|_{C^\ga}\eta_k^{\ga-t'}$.
\end{lem}
The previous lemma, combined with the bound on $\|f_\ep\|_{C^\ga}$ implied by \eqref{eq:BdFepC2}, immediately yields
\begin{align}\label{eq:U11}
 (U11)\leq C_\#\|(\mcl_\om f_\ep)^h\|_{C^\ga}\eta_k^{\ga-t'} \leq C_\#C \|f_\ep\|_{C^\ga} \eta_k^{\ga-t'}
\leq C(\ep) \|f\|\eta_k^{\ga-t'}.
\end{align}

For the second term, we note that $(\E_k -I)(\mcl_\om f_\ep)^s$ is a step function with at most $2\nint_\om$ steps with non-zero value. Also, the change of variables formula shows that for each interval $J\subset I$, one has $|1_J|\leq C_\# m(J)^{\frac{1}{p}-t'}$. Hence, recalling that $\|(\mcl_\om f)^s\|_\infty\leq 2\nint_\om^2\|f\|_\infty$, we get
\begin{align}\label{eq:U12}
(U12)\leq 4 \nint_\om^3 \|f_\ep \|_\infty \sup_{1\leq j \leq k} |1_{B_k}| \leq 
C_\# \nint_\om^3 \|f_\ep \|_\infty \eta_k^{\frac{1}{p}-t'}
\leq 
C_\# \nint_\om^3 C(\ep)\|f\| \eta_k^{\frac{1}{p}-t'},
\end{align}
where the last inequality follows once again from \eqref{eq:BdFepC2}.

Combining \eqref{eq:U2}, \eqref{eq:U11} and \eqref{eq:U12} into \eqref{eq:UlamError}, we get
\begin{align}
  |(\mcl_{k,\om}-\mcl_\om)f| &\leq C_\# \Big( \nint^3 C(\ep) \eta_k^{\big(\min(\ga, \frac{1}{p})-t' \big)}  + \ep^{\frac{t-t'}{2}} \Big) \|f\|,
\end{align}
where $\nint$ is the uniform bound on number of branches of $T_\om$, coming from \eqref{it:M1}.
Choosing $\ep$ the infimum such that $\nint^3 C(\ep)\leq \eta_k^{-\frac{1}{2}(\min(\ga, \frac{1}{p})-t')}$ provides $\tau_k$ as desired.
\qed \end{proof}
\qed \end{proof}

\subsection{Convolution-type perturbations}\label{Ssec:PertByConvolution}
In this section we make use of the equivalent norm on $\hpt$ described in Remark~\ref{rmk:EquivHptNorms}, identifying functions in $\hpt$ (and their norms) with functions in $\hptorus$. 

We consider perturbations of non-autonomous maps that arise from convolution with non-negative kernels $Q_k\in L^1(m)$, with $\int Q_k \,dm=1$. They give rise to transfer operators as follows.
\begin{equation}\label{eq:pertByConv}
  \mcl_{k,\om} f(x):= \int \mcl_\om f(y) Q_k(x-y) dy.
\end{equation}
They model at least two interesting types of perturbations: 
\begin{enumerate}
\item Small iid noise. In this case, $Q_k$, supported on $[-\frac{1}{k},\frac{1}{k}]$, represents the distribution of the noise, which is added after applying the corresponding map $T_\om$.
See e.g. \cite[\S3.3]{BaladiBook} for details.
\item Ces\`aro averages of Fourier series. In this case, $Q_k$ is the Fej\'er kernel $Q_k(x)=\frac{\sin(\pi kx)^2}{k\sin(\pi x)^2}$, and $Q_k*f=\frac{1}{k}\sum_{j=0}^{k-1} S_j(f)$, where $S_k(f)(x) =\sum_{j=-k}^k \hat{f}(j) e^{2\pi i j x}$ is the truncated Fourier series of $f$.
\end{enumerate}
\begin{rmk}
We point out that the Galerkin projection on Fourier modes, corresponding to truncation of Fourier series, is obtained from convolution with Dirichlet kernels, which are not positive. Although a convergence result in this case remains open, the numerical behavior appears to be good as well. 
This is illustrated in \S\ref{sec:numEx}.
\end{rmk}

\begin{thm}\label{thm:pertConv}
Let $\mcl$ be a non-autonomous Lasota-Yorke map, as defined in \S\ref{S:LYMaps}, satisfying $\int \log\tld{\al}\, d\bbp<0$.
Let $\{\mcl_k\}_{k\in \N}$ be a family of random perturbations, as in \eqref{eq:pertByConv}, such that $\lim_{k\to \infty}\int Q_k(x)|x|\,dx=0$\footnote{This condition is equivalent to weak convergence of $Q_k$ to $\del_0$. If the contraction condition of Theorem~\ref{thm:pertConv} requires taking either $N>1$  in \eqref{eq:L1BV-LY-N} and/or $\tld{N}>1$ in \eqref{eq:LY-hpt-lp}, the conclusions remain valid provided the convolutions are taken after $\max\{N, \tld{N}\}$ compositions.}.
Then, for sufficiently large $k$, $\mcl_k$ has a unique random acim.
Let us call it $F_k$.
Then, $\lim_{k\to \infty}F_k=F$ fibrewise in $|\cdot|$.
\end{thm}
\begin{proof}
We will show that conditions \eqref{it:TopExp}--\eqref{it:SmallPert} of Theorem~\ref{thm:StabRandomAcim} are satisfied.
\eqref{it:TopExp} is clear. 

Recall from e.g.  \cite[\S I.2]{Katznelson} that $L_p$ is a homogeneous Banach space. That is, for every $\tau\in \T$, $\|f_\tau\|_p= \|f\|_p$, and  $\lim_{\tau\to 0}\|f_\tau-f\|_p=0$, where $f_\tau(x):=f(x-\tau)$.
Hence, it follows from definition \eqref{def:hptorus}
and the fact that $\hat{f_\tau}(j)=e^{-2\pi i \tau j}\hat{f}(j)$ that $\hptorus$ is a homogeneous Banach space.
Thus, $\|Q_k* f \|\leq \|Q_k\|_1\|f\|=\|f\|$ (see e.g. \cite[\S I.2]{Katznelson}). This yields \eqref{it:GenUnifLY} and \eqref{it:PowBd}. 

In view of Remark~\ref{rmk:UnifBdNorm}, \eqref{it:SmallPert} may be checked as follows.
Let $f_\epsilon$ be as in \eqref{eq:fep}.
For each non-negative $Q\in L^1(m)$, with $\int Q \,dm=1$,  we have
\begin{equation}
|(I-Q)*f|=|(I-Q)*(f-f_\epsilon)|+|(I-Q)*f_\epsilon|.
\end{equation}
The first term is bounded as $|I-Q|\cdot |f-f_\epsilon|$, which is controlled by Lemma~\ref{lem:regApprox}.
To control the second term, we consider the map $f\mapsto
(I-Q)*f_\epsilon$ as a Fourier multiplier.

We want to compare the weak norm of
$(I-Q)*f_\epsilon=\sum_je^{-\epsilon(1+(2\pi
  j)^2)}(1-\hat Q(j)) \hat{f}(j)\phi_j$
with the strong norm of $f$.
That is we want to compare the $L^p$ norm of
$$\sum_je^{-\epsilon(1+(2\pi j)^2)}(1+(2\pi j)^2)^{(t'-t)/2}(1-\hat Q(j))a_j\phi_j=:\sum_j c_j a_j\phi_j$$
with the $L^p$ norm of $\sum_j a_j\phi_j$. Here $a_j= \langle j\rangle^{t/2} \hat{f}(j)$, so that $\|f\|=\|\sum_j a_j\phi_j\|_p$, by definition.

This is clearly a Fourier multiplier on $L^p$. We estimate its norm via Lemma~\ref{lem:Lpbound}.
The variation of the coefficients $(c_j)$ is overestimated by their $\ell^1$
norm. This is estimated as
\begin{equation}\label{eq:multbound}
\|c_j\|_1\le (2|J|+1)\max_{|j|\le J}|1-\hat Q(j)|+
2\sum_{|j|\ge J}e^{-\epsilon(1+(2\pi j)^2)}(1+(2\pi j)^2)^{(t'-t)/2}.
\end{equation}

We can choose $J$ so that the second term is at most $\delta/(3C_p)$
where $C_p$ is the constant in Lemma~\ref{lem:Lpbound}, so that it
then suffices to force $\max_{|j|\le J}|1-\hat Q(j)|\le
\delta/(3C_p(2|J|+1))$.

Notice that for $j\le J$, $|1-\hat Q(j)|=|\int_{-1/2}^{1/2}Q(x)(1-e^{-2\pi
  ijx})\,dx|\le \int_{-1/2}^{1/2}Q(x)|1-e^{-2\pi ijx}|\,dx
\le 2\pi J\int_{-1/2}^{1/2}Q(x)|x|\,dx$. Hence, for sufficiently small
values of $\int Q(x)|x|\,dx$, we obtain the required estimate, 
$|(I-Q)*f_\epsilon|\leq \del\|f\|$.
\qed \end{proof}

\subsection{Static perturbations}\label{Ssec:DetPert}
In this section, we establish the following application of Theorem~\ref{thm:StabRandomAcim}.

\begin{thm}\label{thm:NstepStab}
Let $\mc{T}$ be a non-autonomous Lasota-Yorke map, as defined in \S\ref{S:LYMaps}.
For each $k\in \N$, let $\{\mc{T}_k\}_{k\in \N}$ be a family of random Lasota-Yorke maps with the same base as $\mc{T}$, satisfying \eqref{it:M1}--\eqref{it:M3}, with the same bounds as $\mc{T}$.
Assume that there exists a sequence $\{\rho_k\}_{k>0}$ with $\lim_{k\to \infty}\rho_k= 0$ such that for $\paeom$, $d_{LY}(T_{k,\om}, T_\om)\leq \rho_k$. 
Let $\mathcal L,\mathcal L_k$ be the corresponding Perron-Frobenius operators associated to $\mc{T}$ and $\mc{T}_k$.

Then, there exist a constant $A_{\mc{T}}$, independent of $n$, and a measurable function $B_{\mc{T}, n}(\om)$ such that for every sufficiently large $k\in \N$, and $\paeom$,
\begin{equation}\label{eq:StrongLYi}
  \|\mcl_{k,\om}^{(n)}f\|_{\hpt} \leq   A_{\mc{T}}  \nint^{n(1-\frac{1}{p})} \mu^{-n(1+t-\frac{1}{p})} \|f\|_{\hpt} + B_{\mc{T}, n}(\om) \|f\|_p.
\end{equation}
Let $\tld{\al}_n:=A_{\mc{T}}  \nint^{n(1-\frac{1}{p})} \mu^{-n(1+t-\frac{1}{p})}$, and let $N$ be such that $\int\log\tld{\al}_N \,d\bbp<0$\footnote{The existence of such an $N$ is guaranteed, provided $p>1$ is sufficiently close to 1. If necessary, also enlarge $N$ so that a $(BV, L_1)$ Lasota-Yorke inequality \eqref{eq:L1BV-LY-N}, with $\int\log\al_N \, d\bbp<0$, holds for $\mc{T}^N$.}.

Furthermore, suppose that either
\begin{enumerate}[(i)]
 \item \label{it:NPTP}
$\essinf_{\om\in \Om, 0\leq l< j \leq N} \min_{1\leq i, i' \leq \nint}|T^{(j-l)}_{\sigma^l\om}a_{i,\sigma^l\om}-a_{i',\sig^j\om}| >0$, where $\{a_{1,\om}, \dots, a_{\nint, \om}\}$ are the endpoints of the monotonicity partition of $T_\om$; or
\item \label{it:enoughExp}
$(\Om,\mc{F})$ is a compact topological space equipped with its Borel $\sig$-algebra, the map $\mc{T}:\Om \to \LY^{1+\ga}$ is continuous, and $\mu^\ga>2$, where $\mu$ is as in \eqref{it:M3}, and $\ga\leq 1$ is the H\"older exponent of $DT_\om$.
\end{enumerate}

Then, for every sufficiently large $k$, $\mcl_k$ has a unique random acim.
Let $\{F_k\}_{k\in \N}$ be the sequence of random acims for $\mcl_k$.
Then, $\lim_{k\to \infty}F_k=F$ fibrewise in $|\cdot|$.
\end{thm}

\begin{proof}
We will show that there exists $n\in\N$ such that $\mc{T}^{(n)}$ and $\mc{T}_k^{(n)}$ satisfy the hypotheses of Theorem~\ref{thm:StabRandomAcim}, with $n=N$ in case~\eqref{it:NPTP}.

\eqref{it:TopExp} is easy to verify for $\mc{T}^{(n)}$ and $\mc{T}_k^{(n)}$.
Condition~\eqref{it:PowBd} (with respect to $|\cdot|_{\hptt}$) will follow via Lemma~\ref{lem:PowerBoundStrongNorm}, exactly as explained in the bootstrapping argument of \S\ref{S:LYMaps}, once \eqref{it:GenUnifLY} is established (this will be the last step of this proof).
Condition \eqref{it:SmallPert} follows from the next proposition.
\begin{prop}\label{prop:smallPertDeterm}
Assume that there exists a sequence $\{\rho_k\}_{k>0}$ as in the statement of Theorem~\ref{thm:NstepStab}.
Then, for each $n\in \N$, there exists a sequence $\{\tau_k\}_{k>0}$ with $\lim_{k\to \infty}\tau_k= 0$ such that
\[
\sup_{\|g\|=1} |(\mcl^{(n)}_\om-\mcl^{(n)}_{k,\om})g|\leq \tau_k.
\]
\end{prop}

\begin{proof}[Proof of Proposition~\ref{prop:smallPertDeterm}]
We will establish the claim for $n=1$. The general case of fixed $n>1$ follows immediately from the identity
\[
 \mcl^{(n)}_\om -\mcl_{k,\om}^{(n)}=
\sum_{j=0}^{n-1}\mcl^{(j)}_{\sig^{n-j}\om}(\mcl_{\sig^{n-j-1}\om}-\mcl_{k,\sig^{n-j-1}\om}) \mcl_{k,\om}^{(n-j-1)},
\]
and the fact that $\esssup_{\om\in \Om,\, k \in \N,\, 0\leq j <n} \|\mcl_{k,\om}^{(j)}\|$ and $\esssup_{\om\in \Om,\, 0\leq j <n} |\mcl_{\om}^{(j)}|$ are finite for every $n\in \N$, because of the uniform assumptions \eqref{it:M1}--\eqref{it:M3}; see Remark~\ref{rmk:UnifBdNorm}. 

Throughout the proof, functions are regarded as being defined on the circle, $\T$, via the identification of Remark~\ref{rmk:EquivHptNorms}.
We start with the following.
\begin{myclaim}\label{lem:C2approx}
Let $g\in C^2$, and suppose $S,T \in \LY$ satisfy the uniform bounds \eqref{it:M1}--\eqref{it:M3} of \S\ref{S:LYMaps}. Then, 
\[
  |(\mcl_T-\mcl_S)g |\leq C_1 d_{LY}(T,S)^{\frac{1}{p}-t'} \|g\|_{C^2}.
\]
\end{myclaim}
The proof of this claim follows from \cite[\S3]{GTQuas}, given the uniformity assumptions for $T,S$.

Let $f\in \hpt$, $\ep>0$, and let $f_\ep$ be defined as in \eqref{eq:fep}. Recalling
\eqref{eq:BdFepC2} and Lemma~\ref{lem:regApprox} from \S\ref{S:approxSmoothFunctions}, we get:
\begin{align*}
  |(\mcl_T-\mcl_S)f| &\leq  |(\mcl_T-\mcl_S)f_\ep| + |(\mcl_T-\mcl_S)(f_\ep-f)| \\
&\leq C_1C(\ep) d_{LY}(T,S)^{\frac{1}{p}-t'} \|f\|+ C_\# C \ep^{\frac{t-t'}{2}} \|f\|,
\end{align*}
where $C$ is an upper bound on $\{|\mcl_\om|\}_{\om \in \Om}$, $C(\ep)$ comes from \eqref{eq:BdFepC2}, and $C_1$ comes from Claim~\ref{lem:C2approx}.
Choosing $\ep$ the infimum such that $C_1 C(\ep)\leq \rho_k^{-\frac{1}{2}(\frac{1}{p}-t)}$ provides $\tau_k$ as desired, concluding the proof of the proposition.
\qed \end{proof}

The rest of the proof is concerned with verifying condition \eqref{it:GenUnifLY}.
We will show the following.
\begin{prop}\label{lem:hptLpNstepLY}
 Let  $\mc{T}$ and $\{\mc{T}_k\}_{k\in \N}$ be as in Theorem~\ref{thm:NstepStab}. Then:

 In case \eqref{it:NPTP}, there exist a constant $A_{\mc{T}}$ and a measurable function $B_{\mc{T},N}(\om)$ such that for every sufficiently large $k\in \N$, and $\paeom$, we have
\begin{equation}
  \|\mcl_{k,\om}^{(N)}f\|_{\hpt} \leq   A_{\mc{T}}  \nint^{N(1-\frac{1}{p})} \mu^{-N(1+t-\frac{1}{p})} \|f\|_{\hpt} + B_{\mc{T},N}(\om) \|f\|_p.
\end{equation}
This yields \eqref{it:GenUnifLY} for $\mc{T}^{(N)}$ and $\{\mc{T}^{(N)}_k\}_{k\in \N}$. 

In case \eqref{it:enoughExp}, there exist constants $A_{\mc{T}}$, independent of $n$, and $B_{\mc{T}, n}$ such that for every sufficiently large $k\in \N$, every $n\in \N$, and $\paeom$, we have
\begin{equation}\label{eq:StrongLYii}
  \|\mcl_{k,\om}^{(n)}f\|_{\hpt} \leq   A_{\mc{T}}  \nint^{n(1-\frac{1}{p})} \mu^{-n(1+t-\frac{1}{p})} 2^n \|f\|_{\hpt} + B_{\mc{T}, n} \|f\|_p.
\end{equation}
In particular, if $\mu^\ga>2$, one may choose $p>1$ sufficiently close to $1$ and $t< \min\{\ga, \frac{1}{p}\}$ sufficiently close to $\ga$, such that if $n$ is sufficiently large,
\eqref{it:GenUnifLY} holds for $\mc{T}^{(n)}$ and $\{\mc{T}^{(n)}_k\}_{k\in \N}$.
\end{prop}

In order to demonstrate this, we shall make use of a characterization of $\hpt$, due to
Strichartz \cite{Strichartz}.
\begin{thm}[Strichartz]\label{thm:Str1}
Let $p>1$ and $0<t<1$, and $f:\R \to \R$ with $\spt{f}\subseteq [0,1]$. Then $f\in\hpt$ if and only if
$\|f\|_p+\|D_tf\|_p<\infty$, and the implied norm is equivalent to the
standard $\hpt$ norm, where $D_tf$ is given by
\begin{equation}\label{eq:Dt}
D_tf(x)=\lim_{\ep\to 0}\int_{|y|\geq \ep}\frac{f(x+y)-f(x)}{|y|^{1+t}}dy,
\end{equation}
and the limit is in $L_p$.
\end{thm}

The proof of Proposition~\ref{lem:hptLpNstepLY} relies on the following claim, whose proof is deferred until \S\ref{pf:properSpt}.
\begin{myclaim} \label{lem:properSpt}\ \\
  Let $f \in \hpt$ be such that $\spt(f) \subset [a,b]$. Let $a'<a$, $b'>b$ and $c=\min \{ |a-a'|, |b-b'| \}$. Then,
\[
  \|D_t f - 1_{[a', b']} D_t f \|_p \leq C_\# |b-a|^{1-\frac{1}{p}} c^{\frac{1}{p}-1-t} \|f\|_p.
\] 
Furthermore, if $f_1, \dots, f_M \in \hpt$ are such that $\spt(f_j) \subset [a_j,b_j]$ with $\max_{1\leq j \leq M} \{b_j-a_j\}\leq l$; $a'_j<a_j$, $b'_j>b_j$ are such that  $\min_{1\leq j \leq M} \{ |a_j-a'_j|, |b_j-b'_j| \}\geq c$; and the intersection multiplicity of $\{[a'_j,b'_j] \}_{1\leq j \leq M}$ is $\tld{M}$, then
\[
\Big \|\sum_{j=1}^M D_t f_j \Big \|_{p}^{p} \leq C_\# \tld{M}^{p-1} \sum_{j=1}^M \|D_t f_j \|_p^{p} + C_\#(Ml)^{p-1} c^{1-p-pt} \sum_{j=1}^M \|f_j\|_p^{p},
\]
where the intersection multiplicity of a collection $\mathcal C$ of subsets of a set is given by $\max_{x\in\bigcup \mathcal
C}\#\{C\in\mathcal C\colon x\in C\}$.
\end{myclaim}

\begin{proof}[Proof of Proposition~\ref{lem:hptLpNstepLY}]\ \\
\begin{enumerate}
 \renewcommand{\theenumi}{\Roman{enumi}} 
 \renewcommand{\labelenumi}{\textit{(\theenumi)}.}
 
 \item 
\textit{$(\hpt, L_p)$ Lasota-Yorke inequality for $\mc{T}$.}
Recall that $\| \mcl_\om^{(n)} f\|_{\hpt} \leq C_\# \| \mcl_\om^{(n)} f\|_{p}+ C_\# \| D_t(\mcl_\om^{(n)} f)\|_{p}$, and by definition,
\[
  \mcl_\om^{(n)} f(x)=\sum_{i=1}^{\nint_\om^{(n)}} (1_{I_i} |DT_\om^{(n)}|^{-1} f)\circ \xi_i(x).
\]  
(Although the intervals $I_i$ and inverse branches $\xi_i$ depend on $\om$ and $n$, we do not write this dependence explicitly, unless needed.)

Using changes of variables and the inequality $(\sum_{i=1}^M x_i)^p \leq M^{p-1}\sum_{i=1}^M x_i^p$, a direct calculation yields 
\begin{equation}\label{eq:LpBoundPFOp}
\| \mcl_\om^{(n)} f\|_{p} \leq  \big( C_e(T_\om^{(n)})\big)^{1-\frac{1}{p}}\||DT_\om^{(n)}|^{-1}\|_{\infty}^{1-\frac{1}{p}}\|f\|_p,
\end{equation}
where $C_e(T)$ is the intersection multiplicity of $\{ \overline{T(I_i^T)} \}_{1\leq i \leq
\nint^T}$, named \textit{complexity at the end} in \cite{BaladiGouezel}. 

In order to bound $ \|D_t (\mcl_\om^{(n)}f)\|_{p}$, we will use the following two claims.
\begin{myclaim} \label{it:boundPFhpt.1} 
There exists some $C_{\mc{T}}$ such that for every $u \in \hpt$, 
\begin{align*}
\|(1_{I_i} |DT_\om^{(n)}|^{-1} u)\circ \xi_i\|_{\hpt}^p &\leq C_{\mc{T}} \||DT_\om^{(n)}|^{-1}\|_{\infty}^{p-1} \| u \|^p_p \\
& \quad + C_{\mc{T}} \||DT_\om^{(n)}|^{-1}\|_{\infty}^{p+pt-1}\|u\|_{\hpt}^p.
\end{align*}
\end{myclaim}
\begin{myclaim} \label{it:boundPFhpt.2} 
Let $\eta:\R \to [0,1]$ be a $C^\infty$ function supported in $[-1,1]$. For $m\in \Z$, let $\eta_m(x)=\eta(x-m)$, and suppose that the family $\{\eta_m\}_{m\in \Z}$ forms a partition of unity with intersection multiplicity two. Then, for every $u\in \hpt$,
\[
 \sum_{m\in \Z} \| \eta_m u \|^p_{\hpt} \leq C_\# \|u\|_{\hpt}^p.
\]
For each $r>0$, let $R_r(x)=rx$, and $\eta_{m,r}:=\eta_m \circ R_r$. The set $\{\eta_{m,r}\}_{m\in \Z}$ is again a partition of unity.
Furthermore,
\[
 \sum_{m\in \Z} \|\eta_{m,r} u\|^p_{\hpt} \leq C_\# \big( (1+r^{pt})\|u\|_p^p + \|u\|_{\hpt}^p \big).
\]
\end{myclaim}
We will show these claims in \S\ref{S:pfboundPFhpt.1} and \S\ref{S:pfboundPFhpt.2}, respectively. Now, we proceed with the proof relying on them.
Let us fix $\om \in \Om$, and $\del=\del(\om,n)=\min_{1\leq i \leq \nint_\om^{(n)}} \leb(I_i)$ be the size of the shortest branch of $T_\om^{(n)}$. For every $1\leq i \leq \nint_\om^{(n)}$, let $[\tld{a}_{i}, \tld{b}_{i}]=\overline{T_{\om}^{(n)}(I_i)}$.

Let $\{\eta_m\}_{m\in \Z}:\R \to [0,1]$ be a partition of unity as in Claim~\ref{it:boundPFhpt.2}. We note that $\spt(\eta_{m,r})=r^{-1} \spt(\eta_m)$. Hence,
the intersection multiplicity of the supports of $\{\eta_{m,r}\}_{m\in \Z}$ is also 2. 
For every $i,m$ with $I_i \cap \spt(\eta_{m,r})\neq \emptyset$, let $[\tld{a}_{i,m,r}, \tld{b}_{i,m,r}]=\overline{ T_\om^{(n)}(I_i \cap \spt(\eta_{m,r}))}$ and let $[\tld{a}'_{i,m,r}, \tld{b}'_{i,m,r}]=[\tld{a}_{i,m,r}-\del, \tld{b}_{i,m,r}+\del]$.

By definition of $\del$ and the fact that $T_\om^{(n)}$ is piecewise expanding, we know that $\tld{b}_{i}-\tld{a}_{i}>\del$. Thus,
\[
  [\tld{a}'_{i,m,r}, \tld{b}'_{i,m,r}]\subset [\tld{a}_{i}, \tld{b}_{i}] \cup [\tld{a}_{i}-\del, \tld{b}_{i}-\del] \cup [\tld{a}_{i}+\del, \tld{b}_{i}+\del].
\]
Recall that $C_e(T_\om^{(n)})$ is the intersection multiplicity of $\{ [\tld{a}_{i}, \tld{b}_{i}] \}_{1\leq i \leq \nint_\om^{(n)}}$, that $T_\om^{(n)}|_{I_i}$ is injective, 
and that the intersection multiplicity of the supports of $\{\eta_{m,r}\}_{m\in \Z}$ is 2.
Hence, the intersection multiplicity of $\{[\tld{a}'_{i,m,r}, \tld{b}'_{i,m,r}]\}_{m\in \Z, 1\leq i \leq \nint_\om^{(n)}}$ is at most $6 C_e(T_\om^{(n)})$.

Let $f\in \hpt$. 
For each $r>0$, $m\in \Z$ and $1\leq i \leq \nint_\om^{(n)}$, let 
\[
f_{i,m,r}:= \Big(1_{I_i} |DT_\om^{(n)}|^{-1}  \eta_{m,r} f\Big)\circ \xi_i.
\]
Hence, $\spt (f_{i,m,r}) \subset [\tld{a}_{i,m,r}, \tld{b}_{i,m,r}]$.

Assume that $r\geq 3$. Since $\spt(\eta_{m,r})\subset [r^{-1}(m-1), r^{-1}(m+1)]$, there are at most $r+3\leq 2r$ integers $m$ such that $\spt(\eta_{m,r})\cap [0,1]\neq \emptyset$.
Assume also that $r> 2\del^{-1}$. Then, the support of each $\eta_{m,r}$ intersects at most two intervals $I_i$. Let $\Ga_r:=\{(m,i): f_{i,m,r}\not\equiv 0, 1\leq i \leq \nint_\om^{(n)}, m\in \Z\}$.
Then, $\# \Ga_r \leq 4r$.

Set $r=3\del^{-1}$. (Note that since $\nint_\om^{(n)}\geq 2$, then $\del\leq \frac{1}{2}$ and so $r\geq 6$.) 
Now we apply Claim~\ref{lem:properSpt} with the collection $\{f_{i,m,r}\}$.
By the arguments above, we may choose $M=4r$, $\tld{M}=6C_e(T_\om^{(n)})$, $c=\del$ and $l=2r^{-1}\dist^n$, where $\dist$ is, as in \eqref{it:M2}, an upper bound on $\|DT_{\om'}\|_\infty$ for $\bbp$-almost every $\om'$. We get
\begin{align*}
  & \|D_t (\mcl_\om^{(n)}f)\|_{p}^p = \Big\|\sum_{i=1}^{\nint_\om^{(n)}} \sum_{m\in \Z}  D_t ( f_{i,m,r} ) \Big\|_{p}^p = \Big\| \sum_{(m,i)\in \Ga_r} D_t ( f_{i,m,r} ) \Big\|_{p}^p\\
& \quad \leq C_\# C_e(T_\om^{(n)})^{p-1}  \sum_{(m,i)\in \Ga_r} \big\| D_t (f_{i,m,r}) \big\|_p^p 
+ C_\#   \dist^{n(p-1)} \del^{1-p-pt} \sum_{(m,i)\in \Ga_r} \big\| f_{i,m,r}\big\|_p^p. 
\end{align*}
We recall from Theorem~\ref{thm:Str1} that $\|D_t g\|_p\leq C_\# \|g\|_{\hpt}$ for every $g\in \hpt$.
We now use Claim~\ref{it:boundPFhpt.1} combined with the fact that the support of each $\eta_{m,r}$ intersects at most two intervals $I_i$ to bound the first term; and changes of variables combined with the identity  $\sum_{m,i} 1_{I_i} \eta_{m,r}=1$ in $L_p$ to bound the second term. We get
\begin{align*}
   &\|D_t (\mcl_\om^{(n)}f)\|_{p}^p \\
&\leq C_\mc{T}   C_e(T_\om^{(n)})^{p-1}  \sum_{m\in \Z} \Big( \Big\||DT_\om^{(n)}|^{-1}\Big\|_{\infty}^{p-1}\Big\| \eta_{m,r} f  \Big\|^p_p +\Big\||DT_\om^{(n)}|^{-1}\Big\|_{\infty}^{p+pt-1}\Big\| \eta_{m,r} f \Big\|_{\hpt}^p \Big)\\
&+ C_\#  \dist^{n(p-1)} \del^{1-p-pt} \||DT_\om^{(n)}|^{-1}\|_{\infty}^{p-1} \|f\|_p^p.
\end{align*}
Finally, using Claim~\ref{it:boundPFhpt.2} we get
\begin{align*}
&\|D_t (\mcl_\om^{(n)}f)\|_{p}^p\\
 &\leq  C_{\mc{T}}   C_e(T_\om^{(n)})^{p-1} \Big( \Big\||DT_\om^{(n)}|^{-1}\Big\|_{\infty}^{p-1} \|f\|_p^p + \Big\||DT_\om^{(n)}|^{-1}\Big\|_{\infty}^{p+pt-1} \big( (1+ r^{pt})\|f\|_{p}^p+ \|f\|_{\hpt}^p \big) \Big)
\\
&+ C_\#  \dist^{n(p-1)} \del^{1-p-pt} \||DT_\om^{(n)}|^{-1}\|_{\infty}^{p-1} \|f\|_p^p.
\end{align*}

Combining with \eqref{eq:LpBoundPFOp}, we obtain
\begin{equation}\label{eq:LYunpert}
 \begin{split}
    &\|\mcl_\om^{(n)}f\|_{\hpt}^p \\
&\leq 
C_\mc{T} \cdot \big( C_e(T_\om^{(n)})\big)^{p-1} \||DT_\om^{(n)}|^{-1}\|_{\infty}^{p-1} 
\Big(\||DT_\om^{(n)}|^{-1}\|_{\infty}^{pt} \|f\|_{\hpt}^p + \dist^{n(p-1)} \del^{1-p-pt} \|f\|_p^p \Big).
 \end{split}
\end{equation}

The proof is concluded by 
letting  $A_{\mc{T}}=C_\mc{T}^{\frac{1}{p}}$ and 
\[
B_{\mc{T}, n}(\om)=\Big(C_\mc{T} \big( C_e(T_\om^{(n)})\big)^{p-1} \|DT_\om^{(n)}|^{-1}\|_{\infty}^{p-1} \dist^{n(p-1)} \del(\om, n)^{1-p-pt} \Big)^{\frac{1}{p}}.
\] 

\item \textit{Uniform $(\hpt, L_p)$ Lasota-Yorke inequality for $\mc{T}_k$.}
We extend the previous argument to the perturbed random map.
The main difference arises from the fact that the monotonicity partition for $T_{k,\om}^{(N)}$ may have more elements than that of $T_{\om}^{(N)}$. There will always be
 \textit{admissible} intervals, which can be matched to corresponding ones in the monotonicity partition for $T_{\om}^{(N)}$. There may also be \textit{non-admissible} ones, which may appear when $T^{(j-l)}_{\sigma^l\om}a_{i,\sigma^l\om}=a_{i',\sig^j\om}$
for some $i, i', j, l\in \N, \om\in \Om$. This is exactly as in the case of a single map (see \cite[\S3.3]{BaladiBook}).
We point out that \textit{admissibility} and \textit{non-admissibility} depends on the \textit{reference} map $T_{\om}^{(N)}$.

Condition \eqref{it:NPTP} prevents new branches from being created during the first $N$ steps. That is, there are no non-admissible elements for $T_{\om,k}^{(N)}$ (with respect to $T_{\om}^{(N)}$). 
Hence, the size of the shortest branch of $T_{k,\om}^{(N)}$, $\del(k,\om, N)$, is close to  $\del(\om, N)$ for sufficiently large $k$.
Thus, the argument from the previous step remains applicable for $\mc{T}_k$, for sufficiently large $k$. Noting also that $C_e(T_{k,\om}^{(N)})\leq \nint^N$ yields \eqref{eq:StrongLYi}.

Now we deal with condition \eqref{it:enoughExp}. 
First, for each $\om$ and $n\in \N$, there exists $\tld{\del}(\om,n)>0$ such that if $d_{\LY}(T_{\sig^j\om},S_j)<\tld{\del}(\om,n)$ for each $0\leq j < n$, and the maps $S_j$ satisfy \eqref{it:M1}--\eqref{it:M3} with the same constants as $\mc{T}$,
then for every non-admissible element $\eta$ of the monotonicity partition of $S^{(n)}:=S_{n-1} \circ \dots \circ S_0$ with respect to $T_{\om}^{(n)}$, one has that $\eta\subset \eta' \cup \eta''$, for some $\eta', \eta''$ elements of the monotonicity partition of $T_\om^{(n)}$. That is, all non-admissible intervals for $S^{(n)}$ are small compared to elements of the monotonicity partition of $T_\om^{(n)}$. The upshot of this is that, even though a group of up to $2^n$ branches may arise near each endpoint of the monotonicity partition of $T_\om^{(n)}$ from the perturbation, these groups will remain separated from each other if the perturbation is sufficiently small, depending on $T_{\om}^{(n)}$.

Compactness of $\Om$ and continuity of $\mc{T}$ ensure there exist $\om_1, \dots, \om_M\in \Om$ such that  
$\Om= \cup_{i=1}^M B_{\LY}(T_{\om_i}, \tld{\del}(T_{\om_i},n)/2 )$, where $B_{\LY}(T,\del)$ denotes the ball of radius $\del$ around $T$, measured with respect to $d_{\LY}$.

Let $\rho = \min_{1\leq i\leq M} \tld{\del}(T_{\om_i},n)/2$. Then, if $d_{\LY}(T_\om, T_{k,\om})<\rho$, one can 
follow the argument of the previous step, to obtain a Lasota-Yorke inequality for $T_{k,\om}^{(n)}$, with the main difference being that now more branches of $T_{k,\om}^{(n)}$ may intersect the support of each $\eta_{m,r}$, where $r$ is chosen so that at most two branches of $T_{\om_j}^{(n)}$ intersect the support of each $\eta_{m,r}$, for $1\leq j \leq M$. 
Specifically, $\Ga^k_r:=\{(m,i): f_{k,i,m,r}\not\equiv 0, 1\leq i \leq \nint_{k,\om}^{(n)}, m\in \Z\}$,
where $f_{k,i,m,r}:= \Big(1_{I_{k,i}} |DT_{k,\om}^{(n)}|^{-1}  \eta_{m,r} f \Big)\circ \xi_{k,i}$,
may contain several non-admissible branches of $T_{k,\om}^{(n)}$,
with respect to $T_{\om_j}^{(n)}$, for some $1\leq j \leq M$. Thus, by the previous paragraph, for sufficiently large $r$, $\# \Ga^k_r \leq 2^n 4r$, where the factor $4r$ is as in $\# \Ga_r$ of the previous step. 
This difference contributes a factor of $2^{np}$ on the right-hand side of \eqref{eq:LYunpert}. Making use of the fact that $C_e(T_{k,\om}^{(N)})\leq \nint^N$, \eqref{eq:StrongLYii} is established.\qed
\end{enumerate}
\end{proof}
\qed \end{proof}

\subsection{Numerical examples}\label{sec:numEx}
In this section we provide a brief demonstration that the stability results of \S\ref{Ssec:Ulam} and \S\ref{Ssec:PertByConvolution} can be used to rigorously approximate random invariant densities.

Let $\Omega$ be a circle of unit circumference let the driving system $\sigma:\Omega\circlearrowleft$ be a rigid rotation by angle $\alpha=1/\sqrt{2}$.
For $x\in \Omega$ considered to be a point in $[0,1)$, we define a random map as:
\begin{equation}
\label{mapeg}
T_\omega(x)=\left\{
              \begin{array}{ll}
                3(x-\omega)-2.9(x-\omega)(x-\omega-1/3), & \hbox{$\omega\le x<\omega+1/3$;} \\
                -3(x-\omega)+1-2.9(x-\omega-1/3)(x-\omega-2/3), & \hbox{$\omega+1/3\le x<\omega+2/3$;} \\
                7/3(x-\omega-2/3)+2\omega/9, & \hbox{$\omega+2/3\le x<\omega+1$.}
              \end{array}
            \right..
\end{equation}
Graphs of $T_\omega$ for three different $\omega$ are shown in Figure \ref{threemaps}.
\begin{center}
\begin{figure}[hbt]
  \begin{center}
  \includegraphics[width=17cm]{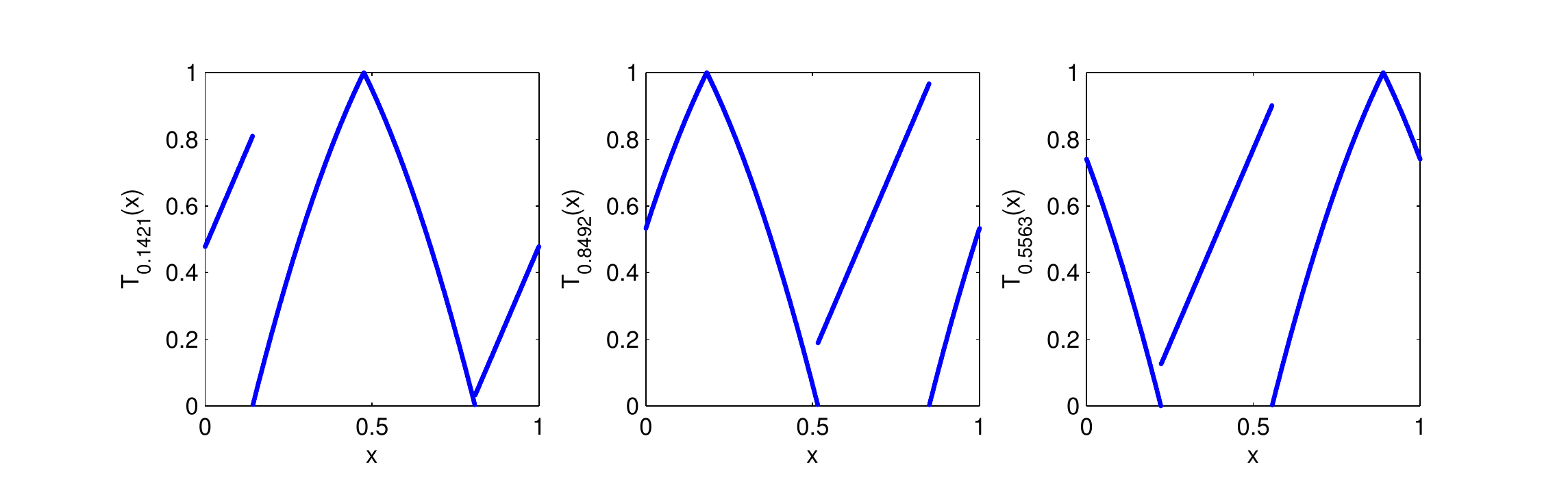}\\
  \caption{Graphs of the maps $T_{\sigma^{20}\omega}, T_{\sigma^{21}\omega}, T_{\sigma^{22}\omega}$, $\omega=0$.}
  \label{threemaps}
  \end{center}
\end{figure}
\end{center}
The graphs $T_\omega$ rotate with $\omega$ and one of three branches is also translated up/down with $\omega$.
The minimum slope of $\{T_\omega\}_{\omega\in\Omega}$ is bounded below by 2.

We employ the Ulam scheme with $k=1000$ (1000 equal subintervals) and a Fej\'{e}r kernel with $k=100$ (100 Fourier modes).
In the Ulam case, we use the well-known formula for the Ulam matrix to construct a matrix representation of $\mathcal{L}_{k,\omega}$:  $[\mathcal{L}_{k,\omega}]_{ij}=m(B_i\cap T^{-1}_\omega B_j)/m(B_j)$, $i,j=1,\ldots,k$, which is the result of Galerkin projection using the basis $\{\mathbf{1}_{B_1},\ldots,\mathbf{1}_{B_k}\}$.
Lebesgue measure in the formula for $[\mathcal{L}_{k,\omega}]_{ij}$ is approximated by a uniform grid of 1000 test points per subinterval, and the estimate of $[\mathcal{L}_{k,\omega}]$ takes less than a second to compute in MATLAB.
In the Fourier case, we first use Galerkin projection onto the basis $\{1,\sin(2\pi x),\cos(2\pi x),\ldots,\sin(2k\pi x),\cos(2k \pi x)\}$.
The relevant integrals are calculated using adaptive Gauss-Kronrod quadrature and we have limited the number of modes to $k=100$ to place an upper limit of 10 minutes of CPU time (on a standard dual-core processor) to calculate the Galerkin projection matrix $[\mathcal{L}'_{k,\omega}]$, representing the projected action of $\mathcal{L}_{\omega}$ on the first $k$ Fourier modes.
We then take a Ces\`{a}ro average to construct $[\mathcal{L}_{k,\omega}]=\frac{1}{k}\sum_{j=0}^{k-1}[\mathcal{L}'_{j,\omega}]$.
Estimates of $f_{k,\sigma^j\omega}, \omega=0, j=21, 21, 22$ are shown in Figure \ref{pushforward}.
\begin{center}
\begin{figure}[hbt]
  \includegraphics[width=17cm]{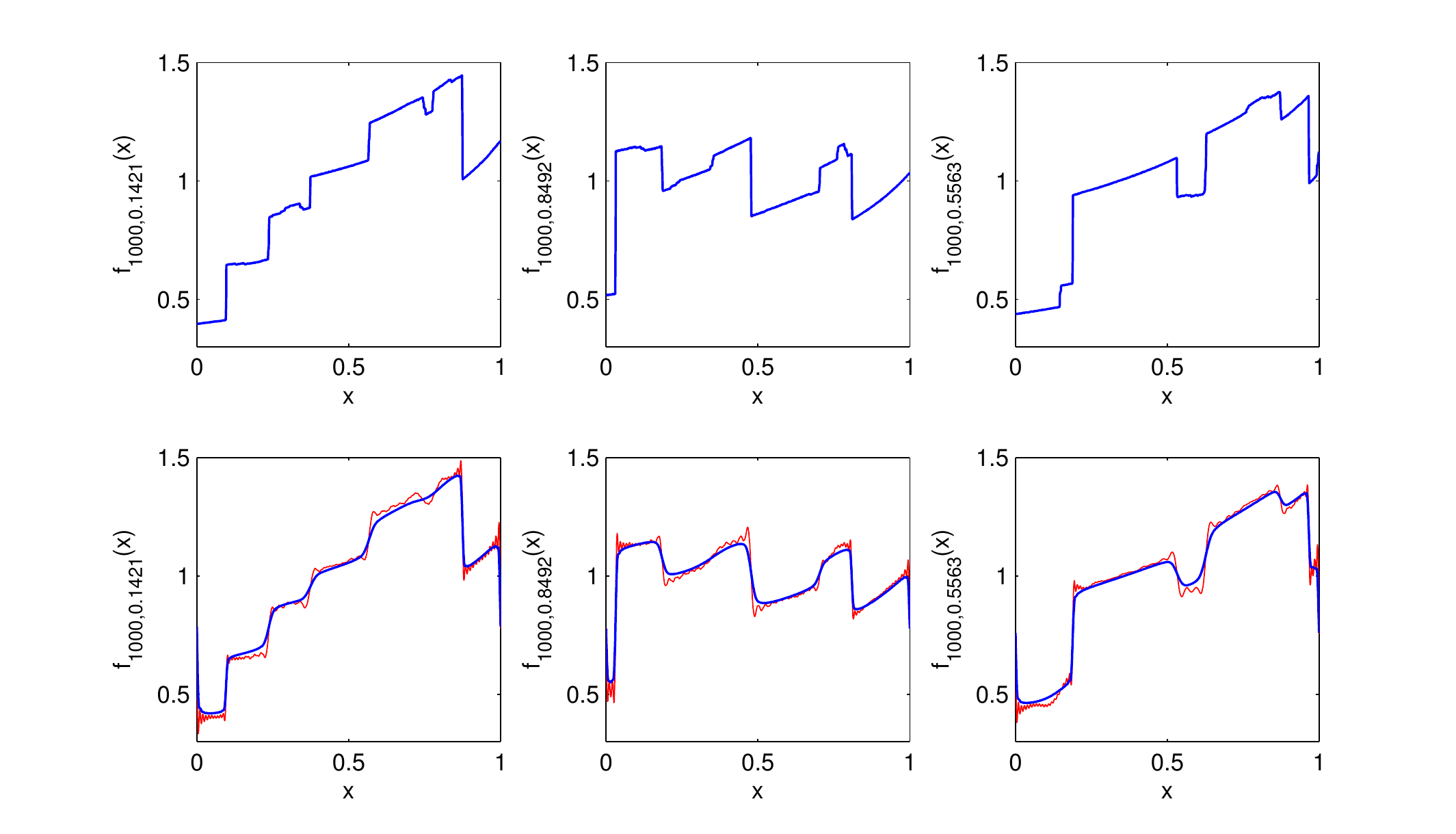}\\
  \caption{Estimates $f_{k,\sigma^j\omega}$, $\omega=0, j=20, 21, 22$ using Ulam's method (upper row) and Fej\'{e}r kernels (lower row, thick [blue, online]). The pure Galerkin estimates using the Galerkin Fourier matrices $[\mathcal{L}'_{k,\omega}]$ are shown as thinner [red, online] curves in the lower row.}\label{pushforward}
\end{figure}
\end{center}
The invariant density estimate $f_{k,\sigma^{20}\omega}$ was created by pushing forward Lebesgue measure (at ``time'' $\omega=0$) by $[\mathcal{L}_{k,\sigma^{19}\omega}]\circ\cdots\circ [\mathcal{L}_{k,\omega}]$, and then pushing two more steps for the estimates $f_{k,\sigma^{21}\omega}$ and $f_{k,\sigma^{22}\omega}$.
By inspecting Figures \ref{threemaps} and \ref{pushforward}, one can see how $\mathcal{L}_{\sigma^j\omega}$ transforms the estimate of $f_{\sigma^j\omega}$ to $f_{\sigma^{j+1}\omega}$, ($j=20,21$), particularly coarse features such as a change in the number of inverse branches.
Though the pure Galerkin estimates are more oscillatory, they appear to pick up more of the finer features than the smoother Fej\'{e}r kernel estimates.
The Ulam estimates are likely the most accurate, given the greater dimensionality of their approximation space.
While the Fourier-based estimates converge slowly in this example (relative to computing time), numerical tests on $C^\infty$ random maps demonstrated rapid convergence, with the Fourier approach taking full advantage of the system's smoothness, to the extent that the influence of modes higher than $k=20$ on the matrix $[\mathcal{L}'_{k,\omega}]$ was of the order of machine accuracy.

\section{Technical proofs}\label{S:techPfs}

\subsection{Proof of Lemma~\ref{lem:regApprox}}\label{pf:regApprox}
We start with a lemma about sequences of bounded variation.
Let $b=(b_j)$, indexed by $\Z$. We define its \emph{variation} by 
$\text{var}(b):=\sum_{j\in\Z}|b_j-b_{j-1}|$.

Let $(\phi_j)$ denote the standard orthonormal basis for $L^2(\mathbb T)$. 
That is, $\phi_j(x):=e^{2\pi i jx}$. 

For a bounded sequence $b$, define an operator on $L^p(\mathbb T)$ by
\begin{equation}\label{eq:defMultiplier}
M_b\colon \sum_j a_j\phi_j\to \sum_j a_jb_j\phi_j.
\end{equation}

\begin{lem}\label{lem:Lpbound}
Let $b=(b_j)$ be a sequence of non-negative reals such that $\text{var}(b)<\infty$
and $b_j\to 0$ as $j\to\pm\infty$.
Then for each $p>1$, $\|M_b\|_p\le C_p\var(b)$.
\end{lem}

The following auxiliary result will be used in the proof.
\begin{lem}\label{lem:comb}
Let $(b_j)$ be as in the statement of Lemma \ref{lem:Lpbound}.
Define sets $\mc{S}_1$ and $\mc{S}_2$ as follows:
\begin{align*}
\mc{S}_1&=\bigcup_j\{j\}\times [0,b_j)\\
\mc{S}_2&=\bigcup_i\bigcup_{j\ge i}\left(\{i,i+1,\ldots,j\}\times [\max(b_{i-1},b_{j+1}),\min(b_i,b_{i+1},\ldots,b_j))\right),
\end{align*}
Then $\mc{S}_1=\mc{S}_2$ and the union in $\mc{S}_2$ is a disjoint union. 
Writing $I_{i,j}$ for 
\[
 [\max(b_{i-1},b_{j+1}),\min(b_i,b_{i+1},\ldots,b_j)),
\]
with the convention that
$[c,d)$ is empty if $d\le c$, and setting $h_{i,j}=|I_{i,j}|$, we have 
$\sum_{i,j}h_{i,j}=\frac12\text{var}(b)$.
\end{lem}

The content of this lemma is illustrated in Figure~\ref{fig:sequence}.
\begin{center}
\begin{figure}[hbt]
\begin{center}
\includegraphics[width=4in]{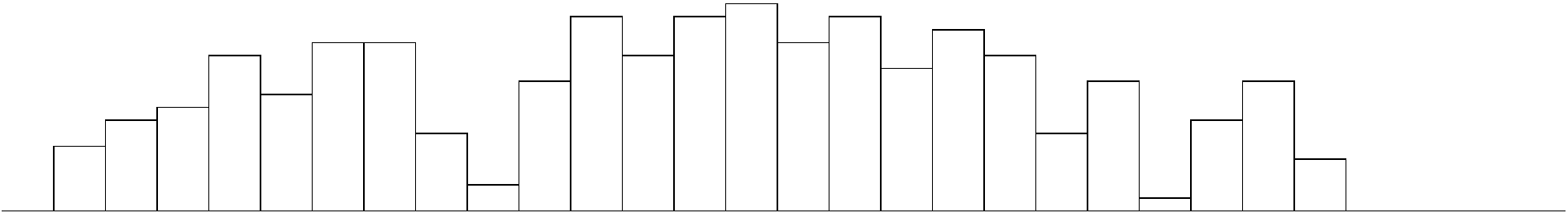}
\vskip3ex
\includegraphics[width=4in]{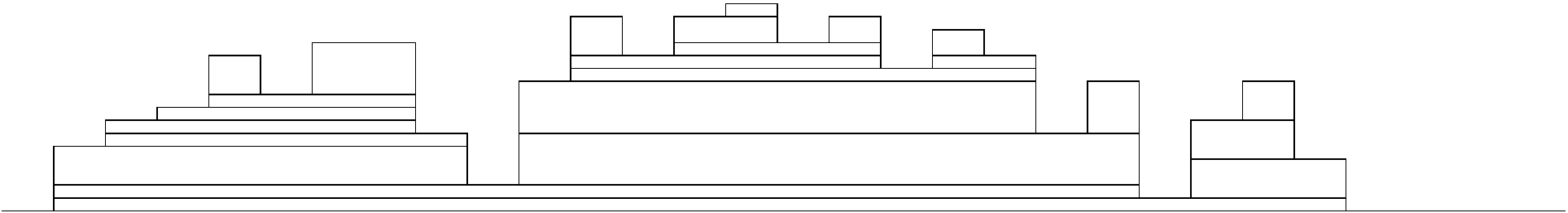}
\caption{Illustration of Lemma~\ref{lem:comb}.}
\label{fig:sequence}
\end{center}
\end{figure}
\end{center}

\begin{proof}[Proof of Lemma \ref{lem:Lpbound}]
Consider $b\colon \Z\to[0,\infty)$ as a function on the integers. By Lemma~\ref{lem:comb}, we can write
\begin{equation*}
b=\sum_{i,j}h_{i,j}\mathbf 1_{[i,j]},
\end{equation*}
where $h_{i,j}$ is given by $|I_{i,j}|$, the length of the interval defined in Lemma~\ref{lem:comb}.
In particular, we deduce
\begin{equation*}
M_b=\sum_{i,j}h_{i,j}S_{[i,j]},
\end{equation*}
where $S_{[i,j]}(f):= \sum_{l=i}^j a_l\phi_l$. 
Thus, $\|M_b\|_p\le \sum_{i,j}h_{i,j}C_p=\frac12\text{var}(b)C_p$, where $C_p$ is a uniform bound on $\|S_k\|_p$, and $S_k$ is the truncated Fourier series  $S_k(f):=\sum_{|j|\leq k} a_j\phi_j$.
\qed \end{proof}

\begin{cor}\label{cor:FourierMultiplier}
Let $(b_j)$ be as in the statement of Lemma \ref{lem:Lpbound}.
 Suppose $(b_j)$ is piecewise monotonic with at most $K$ pieces. Then,
$\|M_b\|_p \leq C_p K \|b\|_\infty$.
\end{cor}

\begin{proof}[Proof of  Lemma~\ref{lem:regApprox}]
Let $\ep>0$, and $b_{j,\ep}=\langle j \rangle^{\frac{t'-t}{2}} (1-e^{-\ep\langle j \rangle})$.
Then, 
\begin{equation}\label{eq:errorRegApprox}
  |f-f_\ep|=\Big\| \sum_{j=1}^\infty b_{j,\ep} \langle j \rangle^{\frac{t}{2}} \hat{f}(j)\phi_j \Big\|_p =\|M_b (J_t f)\|_p \leq \|M_b\|_p\|f\|_{\hpt},
\end{equation}
where $M_b$ is the operator defined in \eqref{eq:defMultiplier},
and $J_t: \hptorus \to L^p(\T)$ is given by $J_t(f):=\mc{F}^{-1}m_t\mc{F}(f)$, with $m_t(\xi)=\langle \xi\rangle^{\frac{t}{2}}=(1+|\xi|^2)^{\frac{t}{2}}$.

For $0<\ga<1$, let $h(x)=x^{-\ga}(1-e^{-\ep x})$. One can check that $h$ has two
intervals of monotonicity. For $x<1/\ep$, one has $h(x)\leq \ep x^{1-\ga}$,
while for $x\ge 1/\ep$, one has $h(x)\leq x^{-\ga}$. In particular, one has
$\|h\|_\infty \leq\ep^{\ga}$.

Using the above with $\ga=\frac{t-t'}{2}$, the lemma follows from \eqref{eq:errorRegApprox} and Corollary~\ref{cor:FourierMultiplier}.\qed
\end{proof}

\subsection{Boundedness of $\E_k$ in $\hpt$.}\label{S:pfBddEnHpt}

In order to demonstrate this, we shall make use of a theorem of
Strichartz \cite{Strichartz}.

\begin{thm}[Strichartz]\label{thm:Strichartz2}
Let $p>1$ and $0<t<1$, and $f:\R \to \R$ with $\spt{f}\subseteq [0,1]$.
Then $f\in\hpt$ if and only if
$\|f\|_p+\|S_tf\|_p<\infty$ and the implied norm is equivalent to the
standard $\hpt$ norm, where $S_tf$ is given by
\begin{equation}\label{eq:St}
S_tf(x)=\left(\int_{0}^{\infty}\frac{dr}{r^{1+2t}}\left(
\int_{-1}^1|f(x+ry)-f(x)|\,dy\right)^2\right)^{1/2}.
\end{equation}
\end{thm}

\begin{proof}[Proof of Lemma \ref{thm:hptbound}]
We shall use the notation $A\lesssim B$ to indicate that the
quantity $A$ is bounded by a constant multiple of the quantity $B$,
where the constant is independent of $k$ and any function to which the
inequality is being applied.

Let $k$ be fixed (although we ensure that all bounds that we give are
independent of $k$). For $x\in [0,1)$, let $j(x)$ denote the index of
the interval to which $x$ belongs. That is $j(x)=\lfloor kx\rfloor$.

We have $\|\E_kf\|_p\le \|f\|_p$ so it suffices to show that
$\|S_t(\E_kf)\|_p\lesssim \|S_tf\|_p$.

We let $H_rf(x)$ be the outer integrand in $S_tf(x)$, that is
\begin{equation}\label{eq:Hr}
H_rf(x)=\frac 1{r^{1+2t}}\left(\int_{-1}^1|f(x+ry)-f(x)|\,dy\right)^2.
\end{equation}

Notice that $S_tf(x)\le S_t^{(1)}f(x)+S_t^{(2)}f(x)$, where
\begin{align*}
S_t^{(1)}f(x)&=\left(\int_0^{1/(2k)}H_rf(x)\,dr\right)^{1/2}\text{ and }\\
S_t^{(2)}f(x)&=\left(\int_{1/(2k)}^\infty H_rf(x)\,dr\right)^{1/2}.
\end{align*}

We start by establishing an inequality that we use several times. Let
$I_j$ denote the interval $[(j-1)/k,j/k)$. 

\begin{myclaim}\label{claim:f-Ef}
  \begin{equation}\label{eq:claimone}
    \int_0^1|\E_kf(x)-f(x)|^p\, dx
    \lesssim
    k^{-pt}\int_0^1dx\,\left(\int_{2/k}^{3/k}H_rf(x)\,dr\right)^{p/2}.
  \end{equation}
\end{myclaim}

\begin{proof}
  Let $x\in [0,1]$. If $\frac 2k\le r\le \frac 3k$, then 
\begin{align*}
H_rf(x)&\gtrsim k^{1+2t}\left(\int_{-1}^1|f(x+ry)-f(x)|\,dy\right)^2\\
&\gtrsim k^{3+2t}\left(\int_{x-r}^{x+r}|f(s)-f(x)|\,ds\right)^2\\
&\ge k^{3+2t}\left(\int_{I_{j(x)}}|f(s)-f(x)|\,ds\right)^2\\
&\ge k^{3+2t}\left(\int_{I_{j(x)}}|\E_kf(s)-f(x)|\,ds\right)^2\\
&=k^{1+2t}|\E_kf(x)-f(x)|^2.
\end{align*}
Integrating in $r$ over the range $[\frac2k,\frac 3k]$ and raising to
the $p/2$ power, we obtain
\begin{equation*}
|\E_kf(x)-f(x)|^p\lesssim k^{-pt}\left(\int_{2/k}^{3/k}H_rf(x)\,dr\right)^{p/2},
\end{equation*}
which establishes the claim upon integrating with respect to $x$.\qed
\end{proof}

For a function $f(x)$, let $f_j$ denote the value of $\E_kf$ on the
interval $I_j$. Recalling definitions \eqref{eq:Hr} and \eqref{eq:St} of $H_r$ and $S_t$, the above implies the inequality

\begin{equation}\label{eq:f-Ef}
k^{pt}\int_0^1|f_{j(x)}(x)-f(x)|^p\,dx\lesssim \|S_tf\|_p^p.
\end{equation}

Straightforward modifications also establish the inequality
\begin{equation}\label{eq:f-TEf}
k^{pt}\int_0^1|f_{j(x)+1}(x)-f(x)|^p\,dx\lesssim \|S_tf\|_p^p.
\end{equation}

We now estimate $S_t^{(2)}\E_kf(x)$. Letting $r>1/(2k)$, we have
\begin{align*}
&H_r(\E_kf)(x)=\frac{1}{r^{1+2t}}
\left(\int_{-1}^1|\E_kf(x+ry)-\E_kf(x)|\,dy\right)^2\\
&\lesssim\frac{(f(x)-\E_kf(x))^2}{r^{1+2t}}
+ 
\frac{1}{r^{1+2t}}\left(\int_{-1}^1|\E_kf(x+ry)-f(x)|\,dy\right)^2.
\end{align*}

Hence we have
\begin{equation}
\begin{split}\label{eq:splitSt2}
S_t^{(2)}(\E_kf)(x)\lesssim&
\left(|f(x)-\E_kf(x)|^2\int_{1/(2k)}^\infty\frac1{r^{1+2t}}\,dr\right)^{1/2}\\
&+\left(\int_{1/(2k)}^\infty\frac
     {(\int_{-1}^1|\E_kf(x+ry)-f(x)|\,dy)^2}
     {r^{1+2t}}\,dr\right)^{1/2}\\
&\sim k^t|f(x)-\E_kf(x)|+(*),
\end{split}
\end{equation}
where $(*)$ denotes the term on the second line of the inequality.

We then estimate (*) as follows.
\begin{align*}
&\int_{-1}^1|\E_kf(x+ry)-f(x)|\,dy\\
&=\frac{1}{2r}\int_{x-r}^{x+r}|\E_kf(s)-f(x)|\,ds\\
&\lesssim\frac{1}{2r}\sum_{\{j\colon I_j\cap [x-r,x+r]\ne\emptyset\}}
\int_{I_j}|\E_kf(s)-f(x)|\,ds\\
&\le \frac{1}{2r}\sum_{\{j\colon I_j\cap [x-r,x+r]\ne\emptyset\}}\int_{I_j}|f(s)-f(x)|\,ds\\
&\lesssim \int_{-1}^1\left|f\left(x+(r+\tfrac1k)y\right)-f(x)\right|\,dy,
\end{align*}
so that 
$(*)\lesssim\left( \int_{1/(2k)}^\infty H_{r+\frac1k}f(x)\,dr\right)^{1/2}\le S^{(2)}_tf(x)$.
Hence, by Theorem~\ref{thm:Strichartz2} and definition of $S^{(2)}_tf$, $\|(*)\|_p\lesssim \|f\|_{\hpt}$.
Combining this with \eqref{eq:splitSt2} and \eqref{eq:f-Ef}, we deduce
$\|S_t^{(2)}(\E_kf)\|_p\lesssim \|f\|_{\hpt}$.

It remains to show that $\|S_t^{(1)}(\E_kf)\|_p\lesssim \|f\|_{\hpt}$.
Let $x=\frac jk-h$, where we assume $h<1/(2k)$ (the other case being similar).
We have $H_r\E_kf(x)=0$ if $r\le h$ and, recalling that $r\leq 1/(2k)$, $H_r\E_kf(x)\le |f_{j+1}-f_j|^2/r^{1+2t}$ if $r>h$.

Hence 
\begin{align*}
S_t^{(1)}\E_kf(x)&\le |f_{j+1}-f_j|\left(\int_{h}^{1/(2k)}\frac{1}{r^{1+2t}}\,dr\right)^{1/2}\\
&\lesssim |f_{j+1}-f_j|h^{-t}.
\end{align*} 

Integrating the $p$th power, we see
\begin{align*}
\|S_t^{(1)}\E_kf\|_p^p&\lesssim \sum_j|f_{j+1}-f_j|^pk^{pt-1}\\
&\sim k^{pt}\int_0^1 |f_{j(x)+1}-f_j(x)|^p\,dx\\
&\sim k^{pt}\left(\int_0^1 |f_{j(x)+1}-f(x)|^p\,dx+\int_0^1 |f_{j(x)}-f(x)|^p\,dx\right).
\end{align*}

The desired bound then follows from \eqref{eq:f-Ef} and \eqref{eq:f-TEf}.
\qed \end{proof}

\subsection{Proof of Lemma~\ref{lem:UlamError4Holder}}\label{pf:UlamError4Holder}

Let $g\in C^\gamma$ and $t<\min(\gamma,1/p)$. We will show that $\|(\mathbb E_k-1)g\|_{\hpt}
\le C_\#\|g\|_{C^\gamma}k^{t-\gamma}$.

  We use the Strichartz equivalent characterization of $\hpt$ of Theorem~~\ref{thm:Strichartz2} again.
  Let $x\in[0,1]$ be at a distance $s$ from one of the endpoints of 
  the partition of the interval into subintervals of length $1/k$.
  Let $g\in C^\gamma$ and let $h=\mathbb E_kg-g$. We check that
  $|h|(z)\le \|g\|_\gamma k^{-\gamma}$ for all $z$.

  We have $\|h\|_{H^p_t}\approx \|h\|_{p}+\|S_th\|_{p}$ where
\begin{equation*}
  S_th(x)=\left(\int_0^\infty \frac{dr}{r^{1+2t}}
  \left(\int_{-1}^1|h(x+ry)-h(x)|\,dy\right)^2\right)^{1/2}.
\end{equation*}

We split the integration over the ranges $[0,s]$ and $[s,\infty)$:
\begin{align*}
& \left(\int_0^s \frac{dr}{r^{1+2t}}
 \left(\int_{-1}^1|h(x+ry)-h(x)|\,dy\right)^2\right)^{1/2}\\
&=\left(\int_0^s \frac{dr}{r^{1+2t}}
 \left(\int_{-1}^1|g(x+ry)-g(x)|\,dy\right)^2\right)^{1/2}\\
&\le \left(\int_0^s \frac{dr}{r^{1+2t}}\left(\int_{-1}^1
 \|g\|_{C^\gamma}|ry|^{\gamma} \,dy \right)^2 \right)^{1/2}\\
&=C_\#\|g\|_{C^\gamma}\left(\int_0^s dr\ 
 r^{2\gamma-1-2t}\right)^{1/2}\le C_\#\|g\|_{C^\gamma}k^{t-\gamma}.
\end{align*}

Using the uniform bound on $h$, we have
\begin{equation*}
\left(\int_s^\infty \frac {dr}{r^{1+2t}}
\left(\int_{-1}^1|h(x+ry)-h(x)|\,dy\right)^2\right)^{1/2}\\
\le C_\#\|g\|_{C^\gamma}k^{-\gamma}s^{-t}.
\end{equation*}

Since the $L^p$ norm of each part is of the form
$C_\#\|g\|_{C^\gamma}k^{t-\gamma}$, the desired result is obtained.\qquad \qed

\subsection{Proof of Claim~\ref{lem:properSpt}}\label{pf:properSpt}
Let $A=\{ y: |y| \geq c \}$. Using that $\spt(f) \subset [a,b]$, Jensen's inequality and Fubini's theorem we get 
\begin{align*}
  \|D_t f &- 1_{[a', b']} D_t f \|_p^{p} = \int_{\R\setminus [a', b']} \Big | \int \frac{f(x+h) -f(x)}{|h|^{1+t}}dh \Big |^{p} dx \ \\
& \leq \int_{\R\setminus [a', b']} \Big | \int \frac{1_A(h)f(x+h)}{|h|^{1+t}}dh \Big |^{p} \,dx 
 \leq |b-a|^{p-1} \int_{\R\setminus [a', b']} \int \frac{1_A(h)|f(x+h)|^{p}}{|h|^{p+tp}}\,dh \,dx \ \\
& \leq |b-a|^{p-1} \int \frac{1_A(h)}{|h|^{p+tp}} \int_{\R\setminus [a', b']} |f(x+h)|^{p} \,dx \,dh 
 \leq \frac{2|b-a|^{p-1} c^{1-p-pt}}{p+pt-1} \|f\|_p^{p},
\end{align*}
which gives the first part. For the second part, we first apply the triangle inequality to get
\begin{align*}
&\Big \|\sum_{j=1}^M D_t f_j \Big \|_p \leq \Big \|\sum_{j=1}^M  1_{[a'_j, b'_j]} D_t f_j \Big \|_p + \sum_{j=1}^M \|D_t f_j - 1_{[a'_j, b'_j]} D_t f_j\|_p.
\end{align*}
We now use the following inequality, valid for non-negative $x, y$, $(x+y)^{p} \leq 2^{p-1}(x^{p}+y^{p})$,  the fact that the intersection multiplicity of $\{[a'_j,b'_j] \}_{1\leq j \leq M}$ is $\tld{M}$ to bound the $p$-th power of the first sum, and the previous part to bound the $p$-th power of the second sum. We get
\begin{align*}
&\Big \|\sum_{j=1}^M D_t f_j \Big \|_{p}^{p} \leq C_\# \tld{M}^{p-1}\sum_{j=1}^M  \|1_{[a'_j, b'_j]} D_t f_j \|_p^{p} + C_\# l^{p-1}c^{1-p-pt} \Big( \sum_{j=1}^M \|f_j\|_p \Big)^{p}.
\end{align*}
Finally, using the fact that for non-negative numbers $x_i$, we have that $(\sum_{i=1}^M x_i)^{p} \leq M^{p-1}\sum_{i=1}^M x_i^{p}$ to bound the second sum. We obtain
\begin{align*}
&\Big \|\sum_{j=1}^M D_t f_j \Big \|_{p}^{p} \leq  C_\# \tld{M}^{p-1}\sum_{j=1}^M  \|D_t f_j \|_p^{p} + C_\# (Ml)^{p-1}c^{1-p-pt} \sum_{j=1}^M \|f_j\|_p^{p}. \quad \qed
\end{align*}

\subsection{Proof of Claim~\ref{it:boundPFhpt.1}}\label{S:pfboundPFhpt.1}
  Fix $x_0\in \overline{I_i}$ maximizing $|A_i|$, where $A_i:=D\xi_i(x_0)$. Then,
\begin{align*}
\|(1_{I_i} |DT_\om^{(n)}|^{-1} u)\circ \xi_i\|_{\hpt}^p= \|(1_{I_i} |DT_\om^{(n)}|^{-1} u)\circ A_i \circ (A_i^{-1} \circ \xi_i)\|_{\hpt}^p. 
\end{align*}

Using that $A_i$ is linear in the equality and \cite[Lemmas 3.4 and 3.5]{GTQuas}
in the second inequality, we get 
\begin{align*}
&\|(1_{I_i} |DT_\om^{(n)}|^{-1} u)\circ A_i\|_{\hpt}^p \\
&\leq  
C_\# \|(1_{I_i} |DT_\om^{(n)}|^{-1} u)\circ A_i\|_{p}^p+ C_\#\|D_t\big( (1_{I_i} |DT_\om^{(n)}|^{-1} u)\circ A_i \big)\|_{p}^p\\
& \quad = C_\# |A_i|^{-1}\|1_{I_i} |DT_\om^{(n)}|^{-1} u\|^p_p + C_\# |A_i|^{pt-1}\|D_t(1_{I_i} |DT_\om^{(n)}|^{-1} u)\|_{p}^p\\
& \quad \leq C_\# |A_i|^{-1}\| D\xi_i\|^p_{\infty} \|  u\|^p_p + C_\# |A_i|^{p+pt-1}\Big\| \frac{D\xi_i}{A_i} \Big\|^p_{\al} \|u\|_{\hpt}^p.
\end{align*}

Using conditions \eqref{it:M2} and \eqref{it:M3}
and a standard distortion estimate (see e.g. \cite{ManeETDD}), we get that there exists some constant $C_\mc{T}$ such that $\Big\| \frac{D\xi_i}{A_i} \Big\|^p_{\al} \leq C_{\mc{T}}$ for all $n, \om$. Then, by the choice of $A_i$, we get that
\begin{align*}
&\|(1_{I_i} |DT_\om^{(n)}|^{-1} u)\circ A_i\|_{\hpt}^p \\
&\quad \leq C_\# \||DT_\om^{(n)}|^{-1}\|^{p-1}_{\infty} \| u\|^p_p + C_\# C_{\mc{T}}\||DT_\om^{(n)}|^{-1}\|^{p+pt-1}_{\infty} \|u\|_{\hpt}^p.
\end{align*}
Also, there exists $K=K(\mc{R})>0$ such that $\sup_{i, \om, n} \max_{x,x'\in I_{i,\om}^{(n)}} |D\xi_{i,\om}^{(n)}(x)^{-1}D\xi_{i,\om}^{(n)}(x')|<K$. So, in  particular, $|D\xi_i(x)^{-1}A_i|\leq K$ and $|A_i^{-1}D\xi_i(x)|\leq K$ for every $x\in I_i$.
It follows directly from \cite[Lemma 3.7]{GTQuas} 
that if $\phi:\R\circlearrowleft$ is a diffeomorphism such that $\|D\phi\|_\infty, \|D\phi^{-1}\|_\infty\leq K$, then there exists a constant $C_K$ such that for every $u\in \hpt$ we have that
$\|u\circ \phi\|_{\hpt} \leq C_K \|u\|_{\hpt}$.

Combining with the previous estimate we get that
\begin{align*}
&\|(1_{I_i} |DT_\om^{(n)}|^{-1} u)\circ \xi_i\|_{\hpt}^p \\
&\quad \leq C_{\mc{T}} \||DT_\om^{(n)}|^{-1}\|_{\infty}^{p-1}\| u\|^p_p + C_{\mc{T}} \||DT_\om^{(n)}|^{-1}\|_{\infty}^{p+pt-1}\|u\|_{\hpt}^p. \quad \qed
\end{align*}

\subsection{Proof of Claim~\ref{it:boundPFhpt.2}}\label{S:pfboundPFhpt.2}
The first claim follows from \cite[Theorem 2.4.7]{Triebel}.
For the second claim, we first observe that 
\begin{align*}
  \|\eta_{m,r} u\|^p_{\hpt}&=\|(u\circ R_{r^{-1}} \cdot \eta_m) \circ R_r\|^p_{\hpt} \\
&\leq C_\# \Big(\|(u\circ R_{r^{-1}} \cdot \eta_m) \circ R_r\|_{p} + \|D_t((u\circ R_{r^{-1}} \cdot \eta_m) \circ R_r)\|_{p} \Big)^p \\
&\leq C_\# \Big(r^{-1}\|u\circ R_{r^{-1}} \cdot \eta_m\|^p_{p} + r^{-1+pt}\|D_t(u\circ R_{r^{-1}} \cdot \eta_m)\|_{p}^p \Big).
\end{align*}
  
Combining with the first part, we get
\begin{align*}
  \sum_{m\in \Z} \|\eta_{m,r} u\|^p_{\hpt} &\leq 
C_\# \Big(r^{-1} \sum_{m\in \Z}\|u\circ R_{r^{-1}} \cdot \eta_m\|^p_{p} + r^{-1+pt}\sum_{m\in \Z}\|u\circ R_{r^{-1}} \cdot \eta_m\|_{\hpt}^p \Big)\\
&\leq C_\# \Big( \|u \|_p^p +  r^{-1+pt} (\|u \circ R_{r^{-1}}\|_p^p + \|D_t(u\circ R_{r^{-1}})\|_{p}^p) \Big) \\
&\leq C_\# \Big( (1+r^{pt})\|u\|_p^p + \|u\|_{\hpt}^p \Big). \quad \qed
\end{align*}

\subsection*{Acknowledgments}
The authors would like to acknowledge enlightening conversations with Ben Goldys, which led to a convenient approximation scheme used in \S\ref{sec:Ex}, as well as useful discussions with Michael Cowling and Bill McLean, and bibliographical suggestions of Hans Triebel. The research of GF and CGT is supported by an ARC Future Fellowship and an ARC Discovery Project (DP110100068).

\bibliography{RandomUlam_refs,SemiInvOsel_refs}
\bibliographystyle{alpha}
\end{document}